\documentclass[12pt, reqno]{amsart}

\usepackage{amssymb, mathrsfs}
\usepackage{latexsym}
\usepackage{amsmath, commath}
\usepackage{amsthm}
\usepackage{amsfonts}
\usepackage{dsfont}
\usepackage{mathtools}
\usepackage{epsfig}
\usepackage{verbatim}
\usepackage{graphicx}
\usepackage{accents}

\usepackage{stmaryrd}

\usepackage{tikz}
\usetikzlibrary{cd}
\usetikzlibrary{patterns}
\usetikzlibrary{calc}
\usepackage{pgfplots}
\pgfplotsset{compat=1.11}

\usepackage[left=1.4in,right=1.4in]{geometry}
\usepackage{enumerate}
\usepackage[colorlinks=true,linkcolor=blue,citecolor=blue]{hyperref}

\usepackage{soul}

\usepackage{todonotes}

\usepackage[titletoc, title]{appendix}

\usepackage{bm}
\usepackage{bbm}
\usepackage{nicefrac}

\usepackage{slashed}

\newtheorem*{lemma*}{Lemma}

\newtheorem{theorem}{Theorem}[section]
\newtheorem{proposition}[theorem]{Proposition}
\newtheorem{lemma}[theorem]{Lemma}

\newtheorem{conjecture}{Conjecture}
\newtheorem*{open question}{Open Question}

\newtheorem{question}{Question}
\newtheorem*{question*}{Question}

\theoremstyle{definition}
\newtheorem{claim}[theorem]{Claim}
\newtheorem*{claim*}{Claim}
\newtheorem{definition}[theorem]{Definition}

\theoremstyle{remark}
\newtheorem{remark}[theorem]{Remark}

\newcommand{\interior}[1]{\accentset{\mathrm{\scriptscriptstyle o}}{#1}}
\newcommand{\sus}{\mathcal S}

\newcommand{\rko}{\widetilde{KO}}
\newcommand{\sphere}{\mathbb S}
\newcommand{\disk}{\mathbb D}
\newcommand{\Sc}{\textup{Sc}}
\newcommand{\ualg}{\mathcal A}
\newcommand{\uspace}{B}
\newcommand{\ind}{\mathrm{Ind}}
\newcommand{\supp}{\mathrm{Supp}}

\numberwithin{equation}{section}

\begin{document}
\title[On Gromov's compactness question]{On Gromov's compactness question regarding positive scalar curvature}

\author{Shmuel Weinberger}
\address[Shmuel Weinberger]{ Department of Mathematics, University of Chicago }
\email{shmuel@math.uchicago.edu}
\thanks{The first author is partially supported by NSF 2105451.}

\author{Zhizhang Xie}
\address[Zhizhang Xie]{ Department of Mathematics, Texas A\&M University }
\email{xie@math.tamu.edu}
\thanks{The second author is partially supported by NSF 1800737 and 1952693.}
\author{Guoliang Yu}
\address[Guoliang Yu]{ Department of
	Mathematics, Texas A\&M University}
\email{guoliangyu@math.tamu.edu}
\thanks{The third author is partially supported by NSF 1700021, 2000082, and the Simons Fellows Program.}

	\maketitle

	\begin{abstract}
		In this paper, we give both positive and negative answers to Gromov's compactness question regarding positive scalar curvature metrics on noncompact manifolds. First we construct examples that give a negative answer to Gromov's compactness question. These examples are based on the non-vanishing of certain index theoretic invariants that arise at the infinity of the given underlying  manifold.  This is a $ \sideset{}{^1}\varprojlim$ phenomenon and  naturally leads one to conjecture that Gromov's compactness question has a positive answer provided that these $ \sideset{}{^1}\varprojlim$ invariants  also vanish. We prove this is indeed the case for a class of $1$-tame manifolds. 
	\end{abstract}

\section{Introduction}
In the past several years, Gromov has formulated an extensive list of conjectures and
open questions on  scalar curvature \cite{MR3816521, Gromov:2019aa}. This has 
given rise to new perspectives on scalar curvature and inspired a wave of recent activity in this area.  Among his many open questions,  
Gromov proposed the following compactness  question regarding positive scalar curvature on noncompact manifolds in \cite[Section 3.6.1]{Gromov:2019aa}. 

	\begin{question}[Gromov's compactness question]\label{conj:comp}
	Let $X$ be a smooth manifold.  Suppose for any given compact subset $V\subset X$ and any $\rho>0$, there exists a (non-complete) Riemannian metric on $X$ with scalar curvature $\geq 1$ such that the closed $\rho$-neighborhood $N_\rho(V)$ of $V$ in $X$ is compact. Then $X$ admits a complete Riemannian metric with scalar curvature $\geq 1$. 
\end{question}

In this paper, we construct both positive and negative examples for the above compactness question of Gromov. First, we construct negative examples to show that  Gromov's compactness question, as stated in its complete generality, is false.  These negative examples lead us to suggest a modification  (Conjecture \ref{conj:limcomp} in Section \ref{sec:pos}) of Gromov's compactness question. We shall prove the conjecture  for a class of $1$-tame manifolds, hence give a class of positive examples for which Gromov's compactness question  holds true.   

Our first main theorem of the paper is the following. 
\begin{theorem}[cf. Theorem \ref{thm:main}]\label{thm:main-intro}
	There exists a noncompact smooth spin manifold $X$ such that for any given compact subset $V\subset X$ and any $\rho>0$, there exists an incomplete Riemannian metric on $X$ with scalar curvature $\geq 1$ and the closed $\rho$-neighborhood $N_\rho(V)$ of $V$ in $X$ is compact, but $X$ itself does not admit a complete Riemannian metric with uniformly positive scalar curvature. 
\end{theorem}

Let us outline the key steps of our construction of negative examples. The idea behind  this is conceptually simple. In the de Rham cohomology of a noncompact manifold $X$, a closed form $\omega$ which is exact on every compact submanifold of $X$ is itself exact, but this fails integrally for cochains. In other words, one can make a choice of cochains $\nu_i$ on larger and larger parts of $X$ with $\omega = d\nu_i$, but there is a global constraint for them to be compatible. This is governed by (the $\sideset{}{^1}\varprojlim$ of) the system of cohomology (in one dimension lower) of the submanifolds,  which measures the indeterminacy of the forms $\nu_i$ showing that $\omega$ is exact. The analogous idea for the $K$-theory class associated to the Dirac operator on a spin manifold gives rise to the obstruction that we will use to construct our negative examples, where each positive scalar curvature metric on a compact submanifold of $X$ is an analogue of a cochain $\nu$ satisfying $d\nu = \omega$ on the submanifold above.  

Recall that  for any given  noncompact complete metric spaces $Y$ equipped with some suitable exhaustion,  Chang, Weinberger and Yu introduced in \cite[Section 3]{MR4045309}  the following new index map 
\[ \sigma\colon KO^{lf}_\ast(Y) \to KO_\ast(\ualg(Y)),  \]
where $KO^{lf}_\ast(Y)$ are the locally finite $KO$ groups of $Y$ and  $\ualg(Y)$ is some geometric $C^\ast$-algebra constructed using the given exhaustion on $Y$ (cf. Section \ref{sec:pre}). The $C^\ast$-algebra $\ualg(Y)$ encodes information about the changing nature of the fundamental group as one moves to infinity. A key feature of the index map $\sigma$ is  that it can be used to detect the geometry at infinity  of $Y$. In particular, as an application of their index map, Chang, Weinberger and Yu constructed a noncompact spin manifold $M$ equipped with an exhaustion\footnote{Here we say $\{M_i\}$ is an exhaustion of $M$ if  the $M_i$ are compact codimension zero submanifolds, $M_i\subset \interior{M}_{i+1}$ and $M = \cup_i M_i$. Here $\interior{M}_{i+1}$ is the interior of $M_{i+1}$.  } consisting of codimension zero compact submanifolds $(M_i, \partial M_i)$ with boundary such that each $M_i$ has a metric of positive scalar curvature which is collared at the boundary, but $M$ itself does not have a complete metric of uniformly positive scalar curvature \cite[Theorem 4.3]{MR4045309}. Those examples of Chang-Weinberger-Yu are not quite sufficient to serve as negative examples to Gromov's compactness question. We shall improve upon the methods of Change-Weinberger-Yu to construct a noncompact spin manifold $X$ equipped with an exhaustion consisting of codimension zero submanifolds $(X_i, \partial X_i)$ with boundary such that
\begin{enumerate}[$(1)$]
	\item each $X_i$ has a metric of positive scalar curvature which is collared at the boundary, 
	\item \emph{additionally} each annulus region $A_i$ between $\partial X_{i}$ and $\partial X_{i+1}$ can be obtained from $\partial X_i\times [0, 1]$ by attaching handles\footnote{This is equivalent to saying $\partial X_i$ can be  obtained from $\partial X_{i-1}$ via surgeries of codimension $\geq 2$.} of index $\geq 2$,  
	\item  but $X$ itself does not have a complete metric of uniformly positive scalar curvature. 
\end{enumerate}  
We claim that such an $X$ gives a negative answer to Gromov's compactness question.  Indeed, for any given compact subset $V\subset X$, we have $V\subset \interior{X}_i$ for some $i$, since $\{X_i\}$ is an exhaustion of $X$. Here  $\interior{X}_{i}$ is the interior of $X_{i}$.  By construction,  $(X_i, \partial X_i)$ admits a metric $g_i$ of  
scalar curvature $\geq 1$, which is collared at the boundary. For any $\rho>0$, by stretching\footnote{Since $g_i$ has product structure near the boundary, such a stretching does not change the scalar curvature of $g_i$.} the collar neighborhood of $\partial X_i$ if necessary, we can assume $N_\rho(V) \subset X_i$.   Now it only remains to show that for each metric $g_i$ on $X_i$ as above,  we can extend $g_i$  to an incomplete Riemannian metric on $X$ with scalar curvature $\geq 1$, which  is a consequence of the following proposition.
\begin{proposition}[{cf. Proposition  \ref{prop:ext}}]\label{prop:ext-intro}
	Let $Z$ be a cobordism between two closed smooth  manifolds $\partial_- Z$ and $\partial_+ Z$ such that $Z$  is obtained from $\partial_+ Z\times [0, 1]$ by attaching handles of index $\geq 2$. Given any smooth Riemannian metric $h$ on $\partial_- Z $, for any  $k>0$ and $m \in \mathbb R$, there exists a smooth Riemannian metric $g$ on $Z$ such that 
	\begin{enumerate}[$(1)$]
		\item $g$ extends $h$, that is, \[ g|_{\partial_- Z} = h,   \]
		\item the scalar curvature $g$ satisfies $\Sc(g)_z \geq k$ for all $z\in Z$,
		\item the mean curvature of $\partial_-Z$ satisfies $ H_{g}(\partial_- Z)_x\geq m$, 
		for all $x\in  \partial_-Z$.
	\end{enumerate}
\end{proposition} 
 Indeed, assume by induction we have extended $g_i$ to a Riemannian metric $ g_{i \shortrightarrow j}$ on $X_j$ with scalar curvature $\geq 1$. Since $\partial X_j$ is compact,  the mean curvature of $\partial X_j$ (with respect to $ g_{i \shortrightarrow j}$) is bounded below by some constant, say, $-m$.  By Proposition \ref{prop:ext-intro}, the metric  $h = g_{i\shortrightarrow j}|_{\partial X_j}$ extends to a metric $g_{j\shortrightarrow j+1}$ on the annulus $A_j$ with scalar curvature $\geq 1$ such that the mean curvature of $\partial X_j$ (with respect to $g_{j\shortrightarrow j+1}$) is $\geq m+1$. We glue the metrics $g_{i\shortrightarrow j}$ on $X_j$ and $g_{j\shortrightarrow j+1}$ on $A_j$ to obtain a \emph{continuous} metric on $X_{j+1} = X_j\cup_{\partial X_j} A_j$.  Now Miao's  gluing lemma \cite[Section 3]{MR1982695} implies that there exists a \emph{smooth} Riemannian metric $g_{i\shortrightarrow j+1}$ on $X_{j+1}$ with scalar curvature $\geq 1$ such that  $g_{i\shortrightarrow j+1}$ coincide with $g_{i\shortrightarrow j}\cup g_{j\shortrightarrow j+1}$ away from $\varepsilon$-neighborhood of $\partial X_j$ in $X_j\cup_{\partial X_j} A_j$. By repeating the above extension-and-gluing process inductively, we eventually extend $g_i$ on $X_i$ to an (incomplete) Riemannian metric on $X$ with scalar curvature $\geq 1$. This finishes the outline of our construction of negative examples to Gromov's compactness question. The full details will be given in Section \ref{sec:neg}.  
 
The above negative examples to Gromov's compactness question are based on the non-vanishing of certain $ \sideset{}{^1}\varprojlim$ index theoretic invariants that arise at infinity of the given underlying manifold. This naturally leads us to conjecture that Gromov's compactness question has a positive answer provided that these $ \sideset{}{^1}\varprojlim$ invariants vanish. The precise statement of this conjecture (Conjecture \ref{conj:limcomp}), a modification of Gromov's compactness question, will be given in Section \ref{sec:pos}. As supporting evidence, we prove the conjecture for a class of $1$-tame manifolds, hence give a class of positive examples for which Gromov's compactness question holds true.  In particular, our second main theorem of the paper is the following. 
 
 \begin{theorem}[cf. Theorem \ref{thm:pos}]
 	Let $M$ be a noncompact $1$-tame spin manifold of dimension $ n\geq 6$. Let $\Gamma = \pi_1(M) $ and $G = \pi_1^\infty(M)$. Assume that the unstable relative Gromov-Lawson-Rosenberg conjecture holds for the pair $(\Gamma, G)$.  Suppose for any given compact subset $V\subset M$ and any $\rho>0$, there exists an (incomplete) Riemannian metric on $M$ with scalar curvature $\geq 1$ such that the closed $\rho$-neighborhood $N_\rho(V)$ of $V$ in $M$ is compact. Then $M$ admits a complete Riemannian metric of uniformly positive scalar curvature. 
 \end{theorem} 
Here $\pi_1^\infty(M)$ is the fundamental groupoid at infinity of $M$, cf. Definition \ref{def:tame} and the discussion after. The unstable relative Gromov-Lawson-Rosenberg conjecture will be reviewed in Section \ref{sec:pre}, cf. Conjecture \ref{conj:rGLR}. 

The paper is organized as follows. In Section \ref{sec:pre}, we review the construction of some geometric $C^\ast$-algebras, the unstable relative Gromov-Lawson-Ronsenberg conjecture for positive scalar curvature, and the construction of relative higher index. In Section \ref{sec:neg}, we construct negative examples to Gromov's compactness question. In Section \ref{sec:pos}, we suggest a modification of Gromov's compactness question, by imposing an extra vanishing condition on certain $ \sideset{}{^1}\varprojlim$ invariants. We confirm the conjecture for   a class of $1$-tame manifolds, which in particular gives a class of positive examples for which Gromov's compactness question holds true.
 
\vspace{.5cm}
\textbf{Acknowledgments.} We would like to thank Bernhard Hanke for helpful comments.

 \section{Preliminaries}\label{sec:pre}
 
 In this section, we review the construction from \cite{MR4045309}  of some geometric $C^\ast$-algebras associated with exhaustions of noncompact spaces. The $K$-theory groups of these $C^\ast$-algebras are the receptacle of index theoretic invariants that  detect the geometry at infinity of the underlying spaces, which constitute a key ingredient in our construction of counterexamples to Gromov's compactness question.

 Let  $\psi \colon A\to B$ be a homomorphism between two (real) $C^\ast$-algebras $A$ and $B$. The mapping cone $C^\ast$-algebra $C_\psi$ of $\psi$ is given by
 \[ C_\psi\coloneqq \{ (a, f) \mid a\in A, f\in C_0([0, 1), B)  \textup{ and } f(0) = \psi(a)\}. \]
It follows that we have the following short exact sequence of $C^\ast$-algebras: 
\[  0 \to  \sus{B} \to C_\psi \to A \to 0, \]
where $\sus{B} = C_0((0, 1), B)$ is the suspension $C^\ast$-algebra of $B$.

\begin{definition}
	A homomorphism  $\varphi\colon G\to \Gamma$ between two discrete groups induces a homomorphism of real $C^\ast$-algebras $\varphi_\ast \colon C^\ast_{\max}(G) \to C^\ast_{\max}(\Gamma).$
	We define  $C^\ast_{\max}(\Gamma, G)$ to be  the $7$th suspension $\sus^7C_{\varphi_\ast} \cong C_0(\mathbb R^7)\otimes C_{\varphi_\ast} $ of  the mapping cone $C^\ast$-algebra  $C_{\varphi_\ast}$. 
\end{definition}

 \begin{definition}\label{def:exhaust}
 	Let $(Y, d)$ be a noncompact, complete metric space. Suppose $Y_1\subseteq Y_2 \subseteq Y_3\subseteq \cdots $ is a sequence of connected compact subsets of $Y$. We say $\{Y_i\}$ is an admissible exhaustion if the following are satisfied:
 	\begin{enumerate}[$(1)$]
 		\item $Y = \cup_{i=1}^\infty Y_i$;
 		\item for all $j>i$, the subspace $Y_{i, j} = Y_j - \interior{Y}_i$ is connected, where $\interior{Y}_i$ is the interior of $Y_i$; 
 		\item $d(\partial Y_i, \partial Y_j) \to \infty$, as $|j-i|\to \infty$, where $\partial Y_i = Y_i - \interior{Y}_i$.  	\end{enumerate}
 \end{definition}

Now suppose $\{Y_i; Y_{ij}\}$ is an admissible exhaustion of $Y$ as above. Define $D_i^\ast$ to be the $C^\ast$-algebra inductive limit 
\begin{equation}\label{eq:d-alg}
	D_i^\ast \coloneqq \varinjlim_{j>i} C^\ast_{\max}(\pi_1(Y_j), \pi_1(Y_{ij})) \otimes \mathcal K 
\end{equation}
of the directed system 
\[ \cdots \to  C^\ast_{\max}(\pi_1(Y_j), \pi_1(Y_{i, j})) \otimes \mathcal K \xrightarrow{\iota_j} C^\ast_{\max}(\pi_1(Y_{j+1}), \pi_1(Y_{i,j+1})) \otimes \mathcal K \to \cdots   \]
where the homomorphism $\iota_j$ is induced by the inclusion $(Y_j, Y_{i,j}) \hookrightarrow (Y_{j+1}, Y_{i, j+1})$ and $\mathcal K $ is real $C^\ast$-algebra of compact operators on a real Hilbert space. 

Let 
\[  \prod_{i=1}^\infty D_i^\ast = \big\{ (a_1, a_2, \cdots) \mid a_i\in D_i^\ast \textup{ and } \sup_{i} \|a_i\| < \infty\big\}\]
be the $C^\ast$ product algebra of $D_i^\ast$.  There is a natural homomorphism 
\[ \rho_{i+1}\colon D_{i+1}^\ast \to D_i^\ast \] induced by the inclusions of the spaces 
\[   (Y_j, Y_{i+1, j}) \hookrightarrow (Y_{j}, Y_{i, j}) \]
for $j > i+1$. Let $\rho \colon \prod_{i=1}^\infty D_i^\ast \to \prod_{i=1}^\infty D_i^\ast$ be the homomorphism that maps $(a_1, a_2, \cdots)$ to $(\rho_2(a_2), \rho_3(a_3), \cdots)$. 

\begin{definition}\label{def:ualg}
	With the above notation, we define the $C^\ast$-algebra $\ualg(Y)$ by 
	\[ \ualg(Y)\coloneqq \Big\{ a \in C([0, 1], \prod_{i=1}^\infty D_i^\ast) \mid \rho(a(0)) = a(1)\Big\}  \]
\end{definition} 

Consider the inverse system 
\[  KO_n(D_1^\ast)  \xleftarrow{\ (\rho_2)_\ast \ }  KO_n(D_2^\ast) \xleftarrow{\ (\rho_3)_\ast \ } \cdots     \]
where $(\rho_{i+1})_\ast \colon  KO_n(D_{i+1}^\ast) \to  KO_n(D_i^\ast)$ is the map induced by $\rho_{i+1}\colon D_{i+1}^\ast \to D_i^\ast$. If $\Phi \colon \prod_{i=1}^\infty  KO_n(D_i^\ast) \to \prod_{i=1}^\infty  KO_n(D_i^\ast)$ is defined by 
\[ \Phi(a_i) = (a_i - (\rho_{i+1})_\ast(a_{i+1})),  \]
then by definition the inverse limit $\varprojlim KO_n(D_i^\ast)$  of the above inverse system is simply the kernel $\ker(\Phi) $  of $\Phi$. We have the following Milnor exact sequence (cf.  \cite{Guentner-Yu}): 
\begin{equation}\label{eq:milnor}
0 \to \sideset{}{^1}\varprojlim KO_{n+1}(D_i^\ast) \to KO_n (\ualg(Y)) \to  \varprojlim KO_n(D_i^\ast) \to 0.
\end{equation} 
where by definition $\sideset{}{^1}\varprojlim KO_n(D_i^\ast)$ is  the cokernel  $ \mathrm{coker}(\Phi)  $ of $\Phi$.  
This Milnor exact sequence can be derived from the $KO$-theory long exact sequence associated with the following short exact sequence of $C^\ast$-algebras: 
\[   0 \to \sus{(\prod_{i=1}^\infty D_i^\ast)}\to \ualg(Y) \xrightarrow{\ \varpi\ } \prod_{i=1}^\infty D_i^\ast\to 0 \]
together with the fact that 
\[  KO_n\big(\prod_{i=1}^\infty D_i^\ast\big)  \cong \prod_{i=1}^\infty KO_n(D_i^\ast)\]
when $D_i^\ast$ are stable, that is, $D_i^\ast \cong D_i^\ast\otimes \mathcal K$. Here $\varpi\colon \ualg(Y) \to  \prod_{i=1}^\infty D_i^\ast$ is the evaluation map $\varpi(a)  \coloneqq  a(0)$. 

In \cite[Section 3]{MR4045309}, Chang, Weinberger and Yu defined a natural index map  
\begin{equation}\label{eq:indexmap}
	\sigma\colon KO^{lf}_\ast(Y)  \to KO_\ast(\ualg(Y)) 
\end{equation} 
which can be used to detect the geometry at infinity of $Y$.  As an application of their index map, they constructed a noncompact spin manifold $M$ equipped with an exhaustion consisting of codimension zero submanifolds $(M_i, \partial M_i)$ with boundary such that each $M_i$ has a metric of positive scalar curvature which is collared at the boundary, but $M$ itself does not have a complete metric of uniformly positive scalar curvature \cite[Theorem 4.3]{MR4045309}.  We shall review the construction of this index map $\sigma$ in Section \ref{sec:relative}.

Note that the Milnor exact sequence \eqref{eq:milnor} gives rise to the following commutative diagram: 
\[  \begin{tikzcd}[column sep=1.5em]
	0 \arrow[r] &  \sideset{}{^1}\varprojlim KO_{n+1}(Y_i, \partial Y_i)  \arrow[r] \arrow[d] &  KO_n^{lf} (Y)  \arrow[r] \arrow[d]& \varprojlim KO_n(Y_i, \partial Y_i) \arrow[r]  \arrow[d] &  0 \\
	0 \arrow[r] &  \sideset{}{^1}\varprojlim KO_{n+1}(D_i^\ast) \arrow[r] &  KO_n (\ualg(Y))  \arrow[r]& \varprojlim KO_n(D_i^\ast) \arrow[r]  &  0 
\end{tikzcd} \]

Now let us briefly recall the following unstable relative Gromov-Lawson-Rosenberg conjecture, cf. \cite[Conjecture 2.21]{MR4045309}. 
\begin{conjecture}[Unstable relative Gromov-Lawson-Rosenberg conjecture]\label{conj:rGLR} Let $(N,   \partial N)$ be an $n$-dimensional compact spin manifold with boundary. If the relative higher index of the Dirac operator $D$ of $N$ is zero in $KO_n(\pi_1(N), \pi_1(\partial N))$, then there is a metric of positive scalar curvature on $N$ that is collared near $\partial N$.
\end{conjecture}

Because of the failure of the ordinary unstable Gromov–Lawson-Rosenberg conjecture (\cite[Example 2.2]{MR1632971}), in general,  this statement cannot be true as stated. On the other hand, there are special cases where the unstable (resp. relative)  Gromov-Lawson-Rosenberg conjecture is true. For example, the $\pi-\pi$ case of the unstable relative Gromov-Lawson-Rosenberg conjecture is true for all manifolds $(N, \partial N)$ of dimension $\geq 6$, that is,  the unstable relative Gromov-Lawson-Rosenberg conjecture holds  if $\pi_1(\partial N) \to \pi_1(N)$ is an isomorphism. This follows from the surgery theory for positive scalar curvature of Gromov-Lawson  \cite[Theorem A]{MGBL80b} and Schoen-Yau \cite[Corollary 6]{RSSY79b}, as improved by Gajer \cite{MR962295}. Recall that the unstable Gromov-Lawson-Rosenberg conjecture holds  when  $\pi  = \mathbb Z^k \ast F$ the free product of $\mathbb Z^k$ and $F$, where $\mathbb Z^k$ is the free abelian group of rank $k$ and $F$ is a finitely generated free group (cf. \cite[Corollary 3.8]{MR1321004}). In particular, the unstable relative Gromov-Lawson-Rosenberg conjecture holds if $\pi_1(N) = \{e\}$ is the trivial group and $\pi_1(\partial N) = \mathbb Z^k\ast F$ with $F$ a finitely generated free group (see \cite[Theorem 2.23 \& Corollary 2.24]{MR4045309} for more details).  In our construction of negative examples to Gromov's compactness question,  we shall exploit these special groups for which the unstable relative Gromov-Lawson-Rosenberg conjecture holds.  

\subsection{Relative higher index}\label{sec:relative} At the end of this section, let us review the construction of relative higher index for Dirac operators on spin manifolds with boundary and more generally  relative higher index for complete spin manifolds (relative to complements of compact subsets), cf. \cite[Section 2]{MR4045309}. We will also review  the construction of this index map $\sigma$ from line \eqref{eq:indexmap}.

Let us first briefly recall the definition of some geometric $C^\ast$-algebras. For simplicity, let us assume $X$ is a complete Riemannian manifold and $\mathcal S$ is a Hermitian bundle over $X$. Let $H_X$ be the space  $L^2(X, \mathcal S)$ of $L^2$ sections of $\mathcal S$ over $X$.

\begin{definition}
	Let $T$ be a bounded linear operator acting on $H_X$. 
	\begin{enumerate}[(i)]
		\item The propagation of $T$ is defined to be the nonnegative real number 
		\[ \sup\{ d(x, y)\mid (x, y)\in \supp(T)\},\] where $\supp(T)$ is  the complement (in $X\times X$) of the set of points $(x, y)\in X\times X$ for which there exist $f, g\in C_0(X)$ such that $gTf= 0$ and $f(x)\neq 0$, $g(y) \neq 0$. Here $C_0(X)$ is the algebra of all continuous functions on $X$ which vanish at infinity. 
		\item $T$ is said to be locally compact if $fT$ and $Tf$ are compact for all $f\in C_0(X)$.  
	\end{enumerate}
\end{definition}

\begin{definition}\label{def:localg}
	With the same notation as above, let $\mathcal B(H_X)$ be the algebra of all bounded linear operators on $H_X$.  
	\begin{enumerate}[(i)]
		\item The Roe algebra of $X$, denoted by $C^\ast(X)$, is the $C^\ast$-algebra generated by all locally compact operators in $\mathcal B(H_X)$ with finite propagation.
		\item If $Y$ is a subspace of $X$, then the $C^\ast$-algebra $C^\ast(Y; X)$  is defined to be the closed subalgebra of $C^\ast(X)$  generated by all elements $T$ such that  $\supp(T)$ is within finite distance of $Y\times Y$. 
	\end{enumerate}
\end{definition}

Now suppose $\widetilde X$ is a Galois covering space of  $X$. Denote its deck transformation group by $\Gamma$. Lift the Riemannian metric of $X$  to a Riemannian metric on $\widetilde X$ so that the action of $\Gamma$ on $\widetilde X$ is an isometric action. Also lift the Hermitian bundle $\mathcal S$ over $X$ to a Hermitian bundle $\widetilde {\mathcal S}$ over $\widetilde X$. 

\begin{definition}\label{def:equiroe}
With the above notation,	let $H_{\widetilde X} = L^2(\widetilde X, \widetilde{\mathcal S})$. Denote by $\mathbb C[X]^\Gamma$ the $\ast$-algebra of all $\Gamma$-equivariant locally compact operators of finite propagation in $\mathcal B(H_{\widetilde X})$.  
	\begin{enumerate}
		\item 	We define the $\Gamma$-equivariant Roe algebra $C^\ast(\widetilde X)^\Gamma$ to be the closure of $\mathbb C[X]^\Gamma$ in $\mathcal B(H_{\widetilde X})$. 
		\item If $Y$ is a subspace of $X$, then the $C^\ast$-algebra $C^\ast(\widetilde Y; \widetilde X)^\Gamma$  is defined to be the closed subalgebra of $C^\ast(\widetilde X)^\Gamma$  generated by all elements $T$ in $C^\ast(\widetilde X)^\Gamma$ such that  $\supp(T)$ is within finite distance of $\widetilde Y\times \widetilde Y$, where $\widetilde Y$ is the restriction of the covering space $\widetilde X$ on $Y\subset X$.
		\item We define the maximal $\Gamma$-equivariant Roe algebra $C^\ast_{\max}(\widetilde X)^\Gamma$ to be the completion of $\mathbb C[X]^\Gamma$ under the maximal norm:
	\[  \|a\|_{\max} = \sup_{\phi} \ \big\{\|\phi(a)\| : \textup{ all $\ast$-representations } \phi\colon \mathbb C[X]^\Gamma \to \mathcal B(H') \big\}. \]
For  a subspace $Y$ of $X$,  the maximal version $C^\ast_{\max}(\widetilde Y; \widetilde X)^\Gamma$ is defined similarly. 
	\end{enumerate}
  
\end{definition}

Now let us review the construction of relative higher index. If $(N, \partial N)$ is a compact spin manifold with boundary, then we shall attach an infinite cylinder $[0, \infty) \times \partial N$ to $N$ along $\partial N$, and extend the Riemannian metric of $N$ to a complete Riemannian metric on the resulting manifold. The relative higher index of the Dirac operator on $(N, \partial N)$ will in fact be constructed using this complete manifold (relative to the cylindrical end). So for brevity, let us now assume $X$ is a complete Riemannian spin manifold and $K$ is codimension zero compact submanifold (with boundary) of $X$. Let us denote $\Gamma = \pi_1(X)$ and $G = \pi_1(X-K)$. Here if $X-K$ has more than one connected components, then $\pi_1(X-K)$ should mean the fundamental groupoid of $X-K$, that is, $\pi_1(X-K)$ is the disjoint union  $\coprod_{\alpha=1}^\ell \pi_1(Y_\alpha)$, where $Y_\alpha$ are the components of $X-\interior{K}$. In this case, the maximal group $C^\ast$-algebra of $G = \pi_1(X-K)$ is defined to be 
\[   C_{\max}^\ast(G) = \bigoplus_{\alpha=1}^\ell C_{\max}^\ast(\pi_1(Y_\alpha)),  \]
Let $\iota_\alpha\colon G_\alpha \to \Gamma$ be the group homomorphism induced by the inclusion of spaces $Y_\alpha \hookrightarrow X$. Let $C_{\iota_\ast}$ be the mapping cone $C^\ast$-algebra induced by the homomorphism 
\begin{equation}\label{eq:cone}
\iota_\ast\colon C_{\max}^\ast(G) \otimes \mathcal K\to M_\ell(\mathbb C) \otimes C_{\max}^\ast(\Gamma) \otimes \mathcal K
\end{equation}
given by 
\[  \iota_\ast(a_1\oplus \cdots \oplus a_\ell) =  \begin{psmallmatrix}
(\iota_{1})_\ast a_1 &  & \\
& \ddots & \\ 
& & (\iota_{\ell})_\ast a_\ell
\end{psmallmatrix}, \]
where $M_\ell(\mathbb C)$ is the  algebra of $(\ell\times \ell)$ matrices.    We denote by $C_{\max}^\ast(\Gamma, G)$  the $7$th suspension $\sus^7C_{\iota_\ast} \cong C_0(\mathbb R^7)\otimes C_{\iota_\ast} $ of  the mapping cone $C^\ast$-algebra  $C_{\iota_\ast}$. 
  
In the following, we shall review the construction of the relative higher index of the Dirac operator on $X$ (relative to $K$). For simplicity, we assume 
\[ \dim X \equiv 0 \pmod 8, \] while the other dimensions are completely similar by a standard suspension argument.  Let $\widetilde X$ be the universal covering space of $X$ and $\widetilde D$ the associated Dirac operator on $\widetilde X$. Let $\widetilde Y_\alpha$ be the universal covering space of $Y_\alpha$ whose deck transformation group is $G_\alpha$.

   Choose a normalizing function $f$, i.e.  a continuous odd function $f\colon \mathbb R\to \mathbb R$ such that 
\begin{equation}\label{eq:normalizing}
	\lim_{t\to \pm \infty} f(t) = \pm 1.
\end{equation} 
Throughout this section, assume without loss of generality that we have chosen the normalizing function $f$ so that its distributional Fourier transform has compact support. 
Let  $F= f(\widetilde D)$ be the operator obtained by applying functional calculus to $\widetilde D$. 
Since we are in  the even dimensional case,  $\widetilde D$ has odd-degree with respect to the natural $\mathbb Z/2$-grading on the spinor bundle of $\widetilde X$, that is, 
\[ \widetilde D = \begin{pmatrix}
	0 & \widetilde D^-\\
	\widetilde D^+ & 0 
\end{pmatrix} \]
In particular, it follows that   
\[  F = \begin{pmatrix} 
	0 & U \\
	V& 0
 \end{pmatrix}.  \]
for some operators $U$ and $V$.

We define the following invertible element 
\[ W \coloneqq \begin{pmatrix} 1 & U \\ 0 & 1\end{pmatrix} \begin{pmatrix} 1 & 0 \\  V & 1\end{pmatrix} \begin{pmatrix} 1 &  U\\ 0 & 1 \end{pmatrix}\begin{pmatrix} 0 & -1\\ 1 & 0 \end{pmatrix}. \]
and  form the idempotent 
\begin{equation}\label{eq:index}
	p = W \begin{pmatrix} 1 & 0 \\ 0 & 0\end{pmatrix} W^{-1}. 
	= \begin{pmatrix} UV(2-UV) & (2 - UV)(1-UV) U \\ V(1-UV) & (1-VU)^2\end{pmatrix}.
\end{equation}
Let $\chi$ be the characteristic function on $K$ and denote its lift to $\widetilde X$ by $\widetilde \chi$.  Let $u$ be an invertible element in the matrix algebra of $C_0(\mathbb R^7)^+$ representing a generator of $KO_{1}(C_0(\mathbb R^7)) \cong KO_0(C_0(\mathbb R^8)) = \mathbb Z$. Consider the invertible elements 
\[  \mathcal U  =  u\otimes p + 1\otimes (1-p) \textup{ and }  \mathcal V  =  u^{-1} \otimes p + 1\otimes (1-p)\] 
in $(C_0(\mathbb R^7)\otimes C^\ast_{\max}(\widetilde X)^\Gamma)^+$, where $C^\ast_{\max}(\widetilde X)^\Gamma$ is the $\Gamma$-equivariant maximal Roe algebra of $\widetilde X$ and $(C_0(\mathbb R^7)\otimes C^\ast_{\max}(\widetilde X)^\Gamma)^+$ is the unitization of $C_0(\mathbb R^7)\otimes C^\ast_{\max}(\widetilde X)^\Gamma$. 
For each $s\in [0, 1]$, we define the following invertible element 
\begin{equation}\label{eq:inv}
\mathcal W_s \coloneqq \begin{pmatrix} 1 & (1-s)\widetilde \chi \mathcal U \widetilde \chi \\ 0 & 1\end{pmatrix} \begin{pmatrix} 1 & 0 \\ -(1-s)\widetilde \chi \mathcal V \widetilde \chi& 1\end{pmatrix} \begin{pmatrix} 1 & (1-s)\widetilde \chi \mathcal U\widetilde \chi   \\ 0 & 1 \end{pmatrix}\begin{pmatrix} 0 & -1\\ 1 & 0 \end{pmatrix}.
\end{equation}
and  form the idempotent 
\begin{equation}\label{eq:idempath}
	\mathfrak p_s = \mathcal W_s \begin{pmatrix} 1 & 0 \\ 0 & 0\end{pmatrix} \mathcal W_s^{-1}. 
\end{equation}
By construction, each element $\mathfrak p_s$ lies in $(C_0(\mathbb R^7)\otimes C^\ast_{\max}(\widetilde K; \widetilde X)^\Gamma)^+$. In particular, we have  
\[ KO_i( C^\ast_{\max}(\widetilde K; \widetilde X)^\Gamma) \cong KO_i(C^\ast_{\max}(\widetilde K)^\Gamma) \cong  KO_i(C^\ast_{\max}(\Gamma)). \] 

Now let $Z_\alpha = Y_\alpha\cup_{\partial Y_\alpha} (\partial Y_\alpha \times [0, \infty))$ be the manifold obtained from $Y_\alpha$ by attaching an infinite cylinder. We equip $Z_\alpha$ with  a complete Riemannian metric  that agrees with the Riemannian metric of $X$ in a small neighborhood of $Y_\alpha$.  Let $\widetilde Z _\alpha$ be the universal covering space of $Z_\alpha$.  Note that $\pi_1(Z_\alpha) = \pi_1(Y_\alpha)$. Denote by $\partial\widetilde Y_\alpha$ the restriction of the covering space $\widetilde Z_\alpha$ on $\partial Y_\alpha \subset  Z_\alpha$. We apply the same construction above to the Dirac operator $\widetilde D_\alpha$ on $\widetilde Z_\alpha$ (i.e. by replacing $\widetilde D$ by $\widetilde D_\alpha$ and $\chi$ by the characteristic function of $Z_\alpha - Y_\alpha = \partial Y_\alpha \times [0, \infty)$ in the above construction) and denote by $\mathfrak q_\alpha$ the resulting idempotent for when $s=0$. Since the normalizing function $f$ in line \eqref{eq:normalizing} has compactly supported Fourier transform, it follows that   the idempotent  $\mathfrak q_\alpha$ lies in $ (C_0(\mathbb R^7)\otimes C^\ast_{\max}(\partial \widetilde Y_\alpha;  \widetilde Z_\alpha)^{G_\alpha})^+.$ 

The canonical $(\Gamma, G_\alpha)$-equivariant map $\widetilde Z_\alpha \to \widetilde X$ induces a natural $C^\ast$ homomorphism 
\[ \psi_\alpha \colon C_0(\mathbb R^7)\otimes C^\ast_{\max}(\partial \widetilde Y_\alpha, \widetilde Z_\alpha)^{G_\alpha} \to C_0(\mathbb R^7)\otimes C^\ast_{\max}(\widetilde K; \widetilde X)^\Gamma. \]
Let us define the map  
\[  \psi\colon \bigoplus_{\alpha} C_0(\mathbb R^7)\otimes C^\ast_{\max}(\partial \widetilde Y_\alpha, \widetilde Z_\alpha)^{G_\alpha} \to  M_\ell(\mathbb C) \otimes C_0(\mathbb R^7)\otimes C^\ast_{\max}(\widetilde K; \widetilde X)^\Gamma \]
by setting 
\[ \psi(a_1 \oplus \cdots \oplus a_\ell) = \begin{psmallmatrix}
	\psi_{1}(a_1) &  & \\
	& \ddots & \\ 
	& & \psi_{\ell}(a_\ell)
\end{psmallmatrix}.   \]
Recall that,  if  the Fourier transform $\widehat f$ of $f$ is supported in $(-\varepsilon, \varepsilon)$, then  the Fourier transform $\widehat f_\lambda $ of $f_\lambda$ is supported in $(-\lambda \varepsilon, \lambda\varepsilon)$, where $f_\lambda$ is the normalizing function given by $f_\lambda(t) = f(\lambda t)$.  Hence by replacing the normalizing function $f$ in line \eqref{eq:normalizing} by $f_\lambda$ for some sufficiently small $\lambda >0$ if necessary, we can assume the propagations of $F = f(\widetilde D)$ and $f(\widetilde D_\alpha)$ are very small. Then by a standard finite propagation argument,  it follows from the above construction that 
\[  \psi(\oplus_{\alpha=1}^\ell \mathfrak q_\alpha) = \mathfrak p_0. \]
Furthermore,  it follows from the product formula of higher index that each  $\mathfrak q_\alpha$ is a representative of the higher index class of the Dirac operator $D_{\partial Y_\alpha}$ in $KO_{0}(C_0(\mathbb R^7)\otimes C^\ast_{\max}(\partial \widetilde Y_\alpha)^{G_\alpha}).$ 
To summarize, we obtain the following $K$-class 
\[ (\oplus_{\alpha=1}^\ell \mathfrak q_\alpha, \mathfrak p_s) \in C^\ast_{\max}(\Gamma, G)^+ \cong (C_0(\mathbb R^7)\otimes C_{\iota_\ast})^+,  \]
where $C_{\iota_\ast}$ is the mapping cone $C^\ast$-algebra from line \eqref{eq:cone}.

\begin{definition}\label{def:index}
With the above notation, if $\dim X \equiv 0 \pmod 8$, the relative higher index  $\ind_{\Gamma, G}(D)$ of $D$  is defined to be
	\[ \ind_{\Gamma, G}(\widetilde D):= [(\oplus_{\alpha=1}^\ell \mathfrak q_\alpha, \mathfrak p_s)] - \left[\big( \oplus_{\alpha=1}^\ell \begin{psmallmatrix} 1 & 0 \\0 & 0\end{psmallmatrix}, \begin{psmallmatrix} 1 & 0 \\0 & 0\end{psmallmatrix}\big)   \right] \in KO_0(C^\ast_{\max}(\Gamma, G)). \]
\end{definition} 

By a theorem of Chang, Weinberger and Yu \cite[Theorem 2.18]{MR4045309}, if an $n$-dimensional spin manifold $X$ admits a complete Riemannian metric that has  uniformly positive scalar curvature, then its relative higher index is zero in $KO	_n(C^\ast_{\max}(\Gamma, G))$, that is, 
\[ \ind_{\Gamma, G}(D) = 0 \in KO_n(C^\ast_{\max}(\Gamma, G)).  \]
In fact, by a more careful finite propagation argument (as in \cite{Guo:2020ur}\cite{Xie:2021tm}),  we have the following more refined  quantitative vanishing theorem. 

\begin{theorem}[{\cite{Guo:2020ur}\cite{Xie:2021tm}}]\label{thm:quanvanish}
	 With the same notation as above, suppose $X$ admits a complete Riemannian metric $g$ such that the scalar curvature of $g$ is $\geq 1$ on the $\rho$-neighborhood $N_\rho(K)$ of $K$. Then there exists a universal constant $C>0$ such that if $\rho >C$, then the relative higher index of $D$ is zero in $KO_n(C^\ast_{\max}(\Gamma, G))$, that is, 
	 \[ \ind_{\Gamma, G}(D) = 0 \in KO_n(C^\ast_{\max}(\Gamma, G)).  \]  
\end{theorem} 
\begin{proof}
The essential ingredients of the proof have already been carried out in \cite{Guo:2020ur} and \cite{Xie:2021tm}. 	For the convenience of the reader, we shall sketch a proof here. 

Let $\chi$ be the characteristic function $\chi$ on $K$ that appeared in  the above construction of the index class $\ind_{\Gamma, G}(D)$.  Observe that if we replace  $\chi$ by the characteristic function $\chi_r$  of the $r$-neighborhood $N_r(K)$  of $K$, the resulting new element from the construction is also a representative of the index class $\ind_{\Gamma, G}(D)$.    This can be seen by replacing $ \widetilde \chi \mathcal U \widetilde \chi$ and $\widetilde \chi \mathcal V \widetilde \chi$ by 
\begin{equation}\label{eq:linearpath}
t\widetilde \chi \mathcal U \widetilde \chi + (1-t) \widetilde \chi_r \mathcal U \widetilde \chi_r   \textup{ and } t\widetilde \chi \mathcal V \widetilde \chi + (1-t) \widetilde \chi_r \mathcal V \widetilde \chi_r 
\end{equation} 
in the definition of $\mathcal W_s$ from line \eqref{eq:inv}, with $t\in [0, 1]$, where $\widetilde \chi_r$ is the lift of $\chi_r$ to $\widetilde X$. 

By assumption the scalar curvature of $g$ is $\geq 1$ on $N_\rho(K)$. It follows that 
\[  \|\widetilde D v\| \geq \frac{1}{4} \|v\| \] 
for all smooth sections $v \in C_c^\infty( \interior{N}_\rho(K), \widetilde {\mathcal S}),$ where $ \interior{N}_\rho(K)$ is the interior of ${N}_\rho(K)$. A finite propagation argument shows that, as long as $\rho$ is sufficiently large, we can choose the characteristic function  $\chi_{\rho/2}$ of $N_{\rho/2}(K)$ and an appropriate normalizing function $f$ in line \eqref{eq:normalizing} such that $\mathfrak p_s$ in line \eqref{eq:idempath} becomes a trivial idempotent (cf. \cite[Lemma 3.2 and Appendix A]{Xie:2021tm}). The same argument applies to the construction of $\mathfrak q_\alpha$ above so that each $\mathfrak q_\alpha$ becomes a trivial idempotent (cf. \cite[Proof of Theorem 1.3]{Guo:2020ur}). This completes the proof. 
\end{proof}

At the end of this section, let us  review the construction of the index map (cf. \cite[Section 3]{MR4045309}):
\begin{equation*}
	\sigma\colon KO^{lf}_\ast(Y)  \to KO_\ast(\ualg(Y)). 
\end{equation*}
and also introduce a notion of $\sideset{}{^1}\varprojlim$ higher index. 

An element in $ KO^{lf}_\ast(Y)$ is represented by a smooth open spin manifold $X$ together with a proper continuous coarse map $\varphi\colon X\to Y$. Let $\{Y_i\}$ be an admissible exhaustion on $Y$ as in Definition \ref{def:exhaust}. Without loss of generality,  assume $X$ is equipped with an exhaustion $\{X_i\}$ such that  the $X_i$ are compact connected codimension zero submanifolds with $X_i\subset \interior{X}_{i+1}$ and $X = \cup_i X_i$. Here $\interior{X}_{i+1}$ is the interior of $X_{i+1}$. Furthermore, Without loss of generality, we assume $X$ is equipped with a complete Riemannian metric so that $Y_i\subset \varphi(X_i) \subset Y_{i+1}$.

 Let us denote $\Gamma = \pi_1(X)$ and $G_i = \pi_1(X - X_i)$. Here if $X-X_i$ has more than one connected components, then $\pi_1(X-X_i)$ should mean the fundamental groupoid of $X-X_i$, that is, $\pi_1(X-X_i)$ is the disjoint union  $\coprod \pi_1(Y_{i\alpha})$, where $Y_{i\alpha}$ are the components of $X-X_i$. In this case, the maximal group $C^\ast$-algebra of $G_i = \pi_1(X-X_i)$ is defined to be 
\[   C_{\max}^\ast(G_i) = \bigoplus_{\alpha} C_{\max}^\ast(\pi_1(Y_{i\alpha})).  \]
Let $\iota_\alpha\colon \pi_1(Y_{i\alpha}) \to \Gamma$ be the group homomorphism induced by the inclusion $Y_{i\alpha} \hookrightarrow X$. Let $C_{\iota_\ast}$ be the mapping cone $C^\ast$-algebra induced by the homomorphism 
\begin{equation}\label{eq:mappingcone}
	\iota_\ast\colon C_{\max}^\ast(G_i) \otimes \mathcal K\to M_\ell(\mathbb C) \otimes C_{\max}^\ast(\Gamma) \otimes \mathcal K
\end{equation}
given by 
\[  \iota_\ast(a_1\oplus \cdots \oplus a_\ell) =  \begin{psmallmatrix}
	(\iota_{1})_\ast a_1 &  & \\
	& \ddots & \\ 
	& & (\iota_{\ell})_\ast a_\ell
\end{psmallmatrix}, \]
where $M_\ell(\mathbb C)$ is the  algebra of $(\ell\times \ell)$ matrices.    We denote by $C_{\max}^\ast(\Gamma, G_i)$  the $7$th suspension $\sus^7C_{\iota_\ast} \cong C_0(\mathbb R^7)\otimes C_{\iota_\ast} $ of  the mapping cone $C^\ast$-algebra  $C_{\iota_\ast}$. The inclusions $(X, X-X_{i+1}) \hookrightarrow (X, X- X_i)$ induce $C^\ast$ homomorphisms 
\[ \rho_{i+1} \colon C_{\max}^\ast(\Gamma, G_{i+1}) \to   C_{\max}^\ast(\Gamma, G_i).   \]
Let $\rho \colon \prod_{i=1}^\infty C_{\max}^\ast(\Gamma, G_i) \to \prod_{i=1}^\infty C_{\max}^\ast(\Gamma, G_i)$ be the homomorphism that maps $(a_1, a_2, \cdots)$ to $(\rho_2(a_2), \rho_3(a_3), \cdots)$. By definition, we have 
	\[ \ualg(X)\coloneqq \Big\{ a \in C([0, 1], \prod_{i=1}^\infty C_{\max}^\ast(\Gamma, G_i) ) \mid \rho(a(0)) = a(1)\Big\}.  \]
Now suppose $\widetilde D$ is the  Dirac operator on the universal covering space of $\widetilde X$. By the construction of relative higher index (cf. Definition \ref{def:index}), for each $(X, X - X_i)$, we have the relative higher index of $\widetilde D$ represented by 
\[  [(\oplus_{\alpha} \mathfrak q_{i\alpha}, (\mathfrak p_i)_s)] - \left[\big( \oplus_{\alpha} \begin{psmallmatrix} 1 & 0 \\0 & 0\end{psmallmatrix}, \begin{psmallmatrix} 1 & 0 \\0 & 0\end{psmallmatrix}\big)   \right] \in KO_n(C_{\max}^\ast(\Gamma, G_i) ). \]
For simplicity, let us write 
\[ \mathfrak a_i \coloneqq (\oplus_{\alpha} \mathfrak q_{i\alpha}, (\mathfrak p_i)_s) \textup{ and } \mathfrak b_i \coloneqq \big( \oplus_{\alpha} \begin{psmallmatrix} 1 & 0 \\0 & 0\end{psmallmatrix}, \begin{psmallmatrix} 1 & 0 \\0 & 0\end{psmallmatrix}\big)  \] 
in $ C_{\max}^\ast(\Gamma, G_i)$.

Consider the characteristic functions $\chi_i$ of $X_i$ and $\chi_{i+1}$ of $X_{i+1}$. By applying  a linear path similar to that  from line \eqref{eq:linearpath} in the construction of  relative higher index, we obtain a continuous path of elements $\mathfrak a_i(t)$ with $t\in [0, 1]$ such that 
\[ \mathfrak a_i(0) = \mathfrak a_i \textup{ and } \mathfrak a_i(1) = \rho_{i+1}(\mathfrak a_{i+1}),  \]
which in particular defines a $K$-theory class of $\mathcal A(X)$. 

\begin{definition}
	 The index map $\sigma\colon KO^{lf}_\ast(Y)  \to KO_\ast(\ualg(Y))$ is defined by setting 
	 \[ \sigma([D, \varphi]) =  \varphi_\ast[(\mathfrak a_1(t), \mathfrak a_2(t), \cdots )] - \varphi_\ast\big[ (\mathfrak b_1, \mathfrak b_2, \cdots) \big] \in KO_n(\ualg(Y)), \] 
	where $\varphi_\ast\colon KO_n(\ualg(X)) \to KO_n(\ualg(Y))$ is the homomorphism induced by the map $\varphi\colon X\to Y$. 
\end{definition}

In \cite[Theorem 3.3]{MR4045309},  Chang, Weinberger and Yu showed that if $X$ admits a complete Riemannian metric that has uniformly positive scalar curvature on the whole $X$, then the above index $\sigma([D, \varphi])$ vanishes in  $KO_n(\ualg(Y))$. This for example follows by an argument that is similar to that used in the proof of Theorem \ref{thm:quanvanish}.

Note that we have the following Milnor exact sequence (cf. line \eqref{eq:milnor}): \begin{equation*}
		0 \to \sideset{}{^1}\varprojlim KO_{n+1}(C_{\max}^\ast(\Gamma, G_i)) \xrightarrow{\zeta} KO_n (\ualg(X)) \xrightarrow{\theta}  \varprojlim KO_n(C_{\max}^\ast(\Gamma, G_i)) \to 0.
\end{equation*}
Let $\{X_i\}$ be the exhaustion of $X$ as above. Suppose we are in a special case where the relative higher index of $D$ on $(X, X_i)$ vanishes in $KO_n(C_{\max}^\ast(\Gamma, G_i))$. For example, by Theorem \ref{thm:quanvanish}, such a condition  is satisfied if for each $X_i\subset X$, there is a complete Riemannian metric $g_i$ on $X$ such that the scalar curvature of $g_i$ is  $\geq 1$ on the $\rho$-neighborhood $N_\rho(X_i)$ of $X_i$ for some sufficiently large $\rho>0$. Let $\sigma(D)$ be the index of $D$ in $KO_n (\ualg(X))$. Then in this case, the image of $\sigma(D)$ under the map $\theta$ is zero in  $\varprojlim KO_n(C_{\max}^\ast(\Gamma, G_i))$. It follows $\sigma(D) = \zeta(c)$ for some unique element $c \in  \sideset{}{^1}\varprojlim KO_{n+1}(C_{\max}^\ast(\Gamma, G_i))$ in this case. 

\begin{definition}\label{def:limone}
	When $\theta (\sigma(D)) $ vanishes in  $\varprojlim KO_n(C_{\max}^\ast(\Gamma, G_i))$, we define the $\sideset{}{^1}\varprojlim$ higher index of $D$ to be 
	\[ \sideset{}{^1}\varprojlim \ind_{\Gamma, G}(D) \coloneqq    c \in  \sideset{}{^1}\varprojlim KO_{n+1}(C_{\max}^\ast(\Gamma, G_i)), \] where $c$ is the unique element in $\sideset{}{^1}\varprojlim KO_{n+1}(C_{\max}^\ast(\Gamma, G_i))$ such that $\zeta(c) = \sigma(D)$. 
\end{definition}

\begin{remark}
The definition of  $\mathcal A(X)$ depends on the particular choice of an exhaustion $\{X_i\}$ and may vary if we choose a different exhaustion of $X$. However, the $KO$-theory of $\mathcal A(X)$ is in fact independent of the choice of exhaustion. Indeed, suppose $\{X'_k\}$ is another exhaustion of $X$, then there exists subsequence $\{X_{i_k}\}$ of the exhaustion $\{X_i\}$ such that $X'_k\subset X_{i_k}$ for each $k\geq 1$. If we denote by $\mathcal A(X;  \{X_{i_k}\})$ (resp. $\mathcal A(X; \{X'_{k}\})$) the $C^\ast$-algebra $\mathcal A(X)$ determined by the exhaustion $\{X_{i_k}\}$ (resp. $\{X'_{k}\}$), then there is a natural $C^\ast$ homomorphism
\begin{equation}\label{eq:mapexhaust1}
\mathcal A(X; \{X_{i_k}\})\to  \mathcal A(X; \{X'_{k}\})
\end{equation} induced by the canonical inclusions of  spaces $(X, X - X_{i_k})\hookrightarrow (X, X - X'_k)$. Similarly, there is subsequence $\{X'_{k_i}\}$ of the exhaustion $\{X'_k\}$ such that $X_i\subset X'_{k_i}$ for each $i\geq 1$, which gives a $C^\ast$ homomorphism 
\begin{equation} \label{eq:mapexhaust2} 
	\mathcal A(X; \{X'_{k_i}\})\to  \mathcal A(X; \{X_{i}\}). 
\end{equation}
Consider the following Milnor exact sequence: 
\begin{equation*}
\resizebox{.95\hsize}{!}{$	0 \to \sideset{}{^1}\varprojlim KO_{n+1}(C_{\max}^\ast(\Gamma, G_i)) \to KO_n (\ualg(X; \{X_i\})) \to  \varprojlim KO_n(C_{\max}^\ast(\Gamma, G_i)) \to 0.$}
\end{equation*} 
As both $\varprojlim KO_n(C_{\max}^\ast(\Gamma, G_i))$ and $\sideset{}{^1}\varprojlim KO_{n+1}(C_{\max}^\ast(\Gamma, G_i))$ remain unchanged when passing to (cofinal) subsequences, it follows that $KO_n (\ualg(X; \{X_i\}))$ remains unchanged when passing to subsequences of the exhaustion $\{X_i\}$. Now by passing further to subsequences of $\{X_{i_k} \}$ and $\{X'_{k_i}\}$, it is not difficult to see that the $C^\ast$ homorphisms from line \eqref{eq:mapexhaust1} and \eqref{eq:mapexhaust2} induce isomorphisms at the level of $KO$-theory.  This shows that $KO_\ast(\ualg(X))$ is independent of the choice of exhaustion. 
\end{remark}

\section{Negative answers to Gromov's compactness question}\label{sec:neg}

In this section, we prove our main theorem (Theorem \ref{thm:main-intro}). 
First, let us fix some notation.  Consider an  inverse system of groups
\[ G_0 \xleftarrow{\varphi_1} G_1 \xleftarrow{\varphi_2} G_2   \xleftarrow{\varphi_3} \cdots  \]
where each $\varphi_i$ is surjective. Let $Y_{0}$ be the  cone over  $BG_0$, where $BG_0$ is the classifying space of $G_0$. We define $Y_i$ inductively as follows. Note that the group homomorphism $\varphi_i$ induces  a continuous map $\Phi_i \colon BG_i \to BG_{i-1}$, where $BG_i$ is the classifying space of $G_i$. Let $Y_i$ be the mapping cylinder obtained by gluing $BG_i\times I = BG_i\times [0, 1]$ to $Y_{i-1}$ along $\partial Y_{i-1} = BG_{i-1}$ via the map $\Phi_i$. 

\begin{definition}\label{def:uspace}
We define  $\uspace_G$ to be 
the resulting mapping cylinder of the above infinite composite. Sometimes we say $\uspace_G$ is the classifying space associated to the inverse system $\{G_i, \varphi_i\}$. 
\end{definition}

Throughout the rest of the paper,  we will only use groups $G_i$ of the form $G_i = \mathbb Z^2\ast F_i$,  where $F_i$ is a finitely generated free group and $\mathbb Z^2\ast F_i$ is the free product of $\mathbb Z^2$ with $F_i$. So from now on, for simplicity,  let us assume each $BG_i$ is compact and has been equipped with a complete metric. Clearly, by appropriately stretching each cylinder $BG_i\times I$,  we can equip $\uspace_G$ with a complete metric such that  the sequence $\{Y_i\}$ becomes an admissible exhaustion of $\uspace_G$ in the sense of Definition \ref{def:exhaust}. 	Let $\psi_i\colon (Y_i, \partial Y_{i}) \to (Y_{i-1}, \partial Y_{i-1})$ be the obvious collapse map, that is,  $\psi_i$ is the identity map on $Y_{i-1}\subset Y_i$ and $\psi_i$ collapses $BG_i\times I$ to $\partial Y_{i-1} = BG_{i-1}$ via the map $\Phi_i$. Let us write 
$Y_{i, i+1}$ for the annulus region  $Y_{i+1} - \interior{Y}_{i}$ between $\partial Y_{i}$ and $\partial Y_{i+1}$. Clearly, $Y_{i, i+1}$ is just the mapping cylinder  obtained by gluing $BG_i\times I$ to $BG_{i-1}$ via the map $\Phi_i \colon BG_i \to BG_{i-1}$.

For the above space $B_G$ with the exhaustion $\{Y_i\}$, there is  Milnor exact sequence
\[  0 \to \sideset{}{^1}\varprojlim KO_{n+1}(Y_j, \partial Y_j) \to KO^{lf}_n (\uspace_G) \to  \varprojlim KO_n(Y_j, \partial Y_j) \to 0\]
where $KO^{lf}_n (\uspace_G)$ is the $n$-th locally finite $KO$-homology of $\uspace_G$. In the following, we shall construct an  inverse system of groups
\[ G_0 \xleftarrow{\varphi_1} G_1 \xleftarrow{\varphi_2} G_2   \xleftarrow{\varphi_3} \cdots  \]
such that each $\varphi_i$ is surjective and  $\sideset{}{^1}\varprojlim KO_{n+1}(Y_j, \partial Y_j)$ is nontrivial for the associated classifying space $B_G$ equipped with the exhaustion $\{Y_i\}$.

Let us define $G_0 = \mathbb Z^2$ and  $G_i = \mathbb Z^2 \ast F_i$,  where each $F_i$ is a finitely generated free group. Let \[  \varphi_i\colon  \mathbb Z^2\ast F_i \to   \mathbb Z^2\ast F_{i-1} \]
is a  surjective group homomorphism such that $\varphi_i $ maps the subgroup $\mathbb Z^2 \ast \{e\} $ of  $\mathbb Z^2\ast  F_i$ to the subgroup   $\mathbb Z^2 \ast \{e\} $ of $\mathbb Z^2\ast F_{i-1}$ via  the $\times 3$ map, that is,  
\[ \mathbb Z^2 \to \mathbb Z^2 \textup{ by } (a, b) \mapsto (3a, 3b).  \] 
Clearly, $F_i$ and $\varphi_i$ with the above properties exist. Consider the resulting inverse system of groups 
\begin{equation}\label{eq:sys}
	G_0 \xleftarrow{\varphi_1} G_1 \xleftarrow{\varphi_2} G_2   \xleftarrow{\varphi_3} \cdots 
\end{equation}  
We have the following non-vanishing result for the $\varprojlim^1$ term. 

\begin{proposition}\label{prop:nonzero}
	Let $\uspace_G$ be the classifying space associated to the inverse system given in \eqref{eq:sys} equipped with the exhaustion $\{Y_i\}$.  Then  the group $\varprojlim^1 KO_3(Y_j, \partial Y_j)$  is nontrivial. 
\end{proposition}
\begin{proof}
	Note that each $Y_i$ is contractible. It follows that 
	\[  KO_n(Y_i, \partial Y_i) \cong \rko_{n-1}(\partial Y_i) = \rko_{n-1}(BG_i). \]
	The classifying space $BG_i = B\mathbb Z^2 \vee BF_i$ is the wedge sum of  $B\mathbb Z^2$ and $BF_i$, where $B\mathbb Z^2$ is a $2$-dimensional torus $\mathbb T^2$ and $BF_i$ is a wedge sum of circles. In particular, we have 
	\[     H_{2}(\partial Y_i) = H_2(BG_i) = H_2(B\mathbb Z^2) \cong \mathbb Z \]
	in this case. It follows that  the group homomorphism $ \varphi_i\colon \mathbb Z^2\ast F_i \to \mathbb Z^2\ast F_{i-1}$ induces the following homomorphism on homology: 
	\[  H_2(BG_i)  \cong \mathbb Z\xrightarrow{\times 9} H_2(BG_i)  \cong \mathbb Z. \]
	It follows that $\varprojlim^1 KO_3(Y_j, \partial Y_j)$ contains a copy of 
	$ \widehat{\mathbb Z}_{3}/\mathbb Z$
	as a subgroup, where $\widehat{\mathbb Z}_{3}$ is the group of $3$-adic integers.\footnote{Recall that $\widehat{\mathbb Z}_{3}$ is the inverse limit of the inverse system 
	 \[ \cdots \to \mathbb Z/3^3\mathbb Z  \to \mathbb Z/3^2\mathbb Z \mathbb \to  Z/3\mathbb Z. \]
 Equivalently, $\widehat{\mathbb Z}_{3}$ can also be viewed as the inverse limit of of the inverse system 
 \[ \cdots \to \mathbb Z/9^3\mathbb Z  \to \mathbb Z/9^2\mathbb Z \mathbb \to  Z/9\mathbb Z. \] } In particular, $\varprojlim^1 KO_3(Y_j, \partial Y_j)$ is nontrivial. 
\end{proof}

The next few results use some basic constructions from surgery theory. Let us first recall some standard terminology from surgery theory. Let $M$ be a manifold of dimension $n$. Observe that 
\[   \partial (\sphere^p \times \disk^{q}) = \sphere^p \times \sphere^{q-1} = \partial (\disk^{p+1} \times \sphere^{q-1}). \] 
Given an embedded $ \sphere^p \times \disk^{q} \subset M$ with $p+q = n$, let $M'$ be the manifold obtained by removing the interior of $\sphere^p \times \disk^{q}$ and gluing in a copy of   $\disk^{p+1}\times \sphere^{q-1}$ along $\sphere^p \times \sphere^{q-1}$. In this case, we say $M'$ is obtained from $M$ by a surgery of dimension $p$ (or codimension $q$). 

The trace of a $p$-surgery is given by 
\[  W \coloneqq (M\times I) \cup_{\sphere^p \times \disk^{q}} (\disk^{p+1}\times \disk^q)   \]
which is a cobordism between $M$ and $M'$. In this case, we say $W$ is obtained from $M\times I$ by attaching a handle of index $(p+1)$.

\begin{lemma}\label{lm:surgery}
	Let $W$ be a spin cobordism between two closed spin manifolds  $\partial_-W$ and $\partial_+W$. Assume  both $\partial_+ W$ and $W$ are connected,  and  $\pi_1(\partial_+W) \to \pi_1(W)$ is an  isomorphism. If $\dim W  \geq 6$, then there exists a spin cobordism $W'$ between $\partial_-W$ and $\partial_+W$ such that 
	\begin{enumerate}[$(1)$]
		\item   $\pi_1(\partial_+W) \to \pi_1(W')$ is an   isomorphism, 
		\item  $W'$ is obtained from $(\partial_+W)\times I$ by attaching handles of index $\geq 3$, or  equivalently $\partial_+W$ is obtained from $\partial_-W$ via surgeries of codimension $\geq 3$. Moreover, $W$ can be obtained from $W'$ by finitely many surgeries of codimension $\geq 3$ 	away from the boundary. 
	\end{enumerate} 
\end{lemma}
\begin{proof}
	A proof of this lemma can be for example found in the proof of \cite[Theorem 2.2]{MR866507}. See also  \cite[Proposition 3.1]{MR4184617}. For the convenience of the reader, let us repeat the argument here.  Note that there is an exact sequence of homotopy groups:
	\[  \pi_2(\partial_+W) \to \pi_2(W) \to \pi_2(W, \partial_+W) \to \pi_1(\partial_+W) \to \pi_1(W)\]
	which reduces to 
	\[  \pi_2(\partial_+W) \to \pi_2(W) \to \pi_2(W, \partial_+W) \to 0\]
	since $\pi_1(\partial_+W) \to \pi_1(W)$ is an isomorphism. As both $\partial_+W$ and $W$ are connected and $\pi_1(\partial_+W) \to \pi_1(W)$ is an  isomorphism, it follows that $W$ is obtained from $\partial_+ W$ by attaching finitely many handles of index $\geq 2$. In particular, it follows that $\pi_2(W, \partial_+ W)$ is finitely generated  a $\mathbb Z[\pi_1(\partial_+ W)]$-module. Since $\dim W \geq 6$ (in fact, $\geq 5$ would suffice here), a set of elements of $\pi_2(W)$ that generate $\pi_2(W, \partial_+W) = \pi_2(W)/\pi_2(\partial_+W)$ can be represented by smoothly embedded $2$-spheres which do not intersect the boundary of $W$.  Since $W$ is a spin manifold, these $2$-spheres have trivial normal bundles, and can be removed by surgeries preserving the spin structure and fundamental group.  Since $\pi_2(W, \partial_+ W)$ is finitely generated as   a $\mathbb Z[\pi_1(\partial_+ W)]$-module, only finitely many surgeries are needed  in order to annihilate $\pi_2(W, \partial_+ W)$.  Denote the resulting new cobordism by $W'$. Note that $(W', \partial_+W)$ is $2$-connected.  Now choose a handle decomposition of $W'$, and proceed to eliminate $0$-, $1$- and $2$-handles as in  \cite[Lemma 1]{MR0189048}. Note that the proof of \cite[Lemma 1]{MR0189048} shows that such a process does not require $W'$ to be an $h$-cobordism, only that $\pi_1(\partial_+W) \cong  \pi_1(W')$ and that $\pi_2(W', \partial_+W)=0$. Thus we conclude that  $W'$ can be obtained from $(\partial_+W)\times I$ by attaching handles of index $\geq 3$. Turning this handle decomposition upside down, we see that $(\partial_+W)\times I$ is obtained from $W'$ by attaching handles of index $\leq n-2$, i.e., $\partial_+W$ is obtained from $\partial_-W$ by performing surgeries of codimension $\geq 3$.  Furthermore, since by construction  $W'$ is obtained from $W$ by surgeries of dimension $2$, we see that $W$ can be obtained from $W'$ by surgeries of codimension $3$.
\end{proof}

\begin{lemma}\label{lm:surgery2}
		Let $W$ be a  cobordism between two closed manifolds  $\partial_-W$ and $\partial_+W$ such that both $\partial_+ W$ and $W$ are connected, and  $\pi_1(\partial_+W) \to \pi_1(W)$ is surjective.
		 If $\dim W  \geq 6$, then  $W$ is obtained from $(\partial_+W)\times I$ by attaching handles of index $\geq 2$ (in other words,   $\partial_+W$ is obtained from $\partial_-W$ via surgeries of codimension $\geq 2$). 
\end{lemma}
\begin{proof}
Recall that given any finitely generated group $G$, if  $H$ is normal subgroup of $G$ such that $G/H$ is a finitely presented group, then $H$ is finitely normally generated\footnote{A normal subgroup $H$ of $G$ is finitely normally generated if $H$ is the normal closure of a subgroup generated by finitely many elements.} in $G$.  In particular, since both $\pi_1(\partial_+W)$ and $\pi_1(W)$ are finitely presented and $\pi_1(\partial_+W) \to \pi_1(W)$ is surjective, it follows that the kernel of  $\pi_1(\partial_+W) \to \pi_1(W)$ is finitely normally generated. We present these finitely many elements  by disjoint circles, and consider disjoint two dimensional discs $\{\mathbb D_i^2\}_{1\leq i \leq \ell}$ in $W$ which intersect $\partial_+ W$  in these circles. A small regular neighborhood $V$ of $(\partial_+ W \bigcup \cup_{i=1}^{\ell} \mathbb D_i^2)$ gives a cobordism between $\partial_+ W$ and a new closed manifold, denoted by $N$.  By construction, we may view $\partial_+W$ as  obtained from $N$ via surgeries of codimension $2$.

Let $\interior{V}$ be the interior of $V$. Consider the cobordism $Z\coloneqq W - \interior{V}$ between $N$ and $\partial_- W$, where the map $\pi_1(N) \to \pi_1(Z)$ is an isomorphism.  As both $N$ and $Z$ are connected and $\pi_1(N) \to \pi_1(Z)$ is an  isomorphism, it follows that $Z$ is obtained from $N$ by attaching finitely many handles of index $\geq 2$. 

To summarize, we see that  $W$ is obtained from $(\partial_+W)\times I$ by attaching handles of index $\geq 2$. This finishes the proof.   
\end{proof}

The following theorem is a consequence of Lemma \ref{lm:surgery2}  and  \cite[Theorem 4.2]{MR4045309}.

\begin{theorem}\label{thm:realization}
	Let $\uspace_G$ be the space from Definition \ref{def:uspace}   equipped with the exhaustion $\{Y_i\}$. Then given any $c\in KO_n^{lf}(\uspace_G)$, there is an element $(M, f)\in \Omega^{spin, lf}_\ast(\uspace_G) $ such that $M$ is a noncompact spin manifolds and $f$ is a proper map from $M$ to $\uspace_G$ satisfying the following: 
\begin{enumerate}[$(1)$]
	\item $f_\ast[D_M] = c$;
	\item the inverse images $(M_i, \partial M_i) = f^{-1}(Y_i, \partial Y_i)$ are compact manifolds with boundary such that the induced maps $\pi_1(M_i)\to \pi_1(Y_i)$ and $\pi_1(\partial M_i) \to \pi_1(\partial Y_i)$ are all isomorphisms;
	\item the induced maps $\pi_1(A_i)\to \pi_1(Y_{i,i+1})$ are isomorphisms for all $i\geq 0$, where $A_i = M_{i+1}- \interior{M}_{i}$ is the annulus region between $\partial M_{i}$ and $\partial M_{i+1}$;
	\item each $A_i$  is obtained from $(\partial M_{i+1})\times I$ by attaching handles of index $\geq 2$. Equivalently, $\partial M_{i+1}$ is obtained from $\partial M_{i}$ via surgeries of codimension $\geq 2$. 
\end{enumerate}	
\end{theorem}
\begin{proof}
	By \cite[Theorem 4.2]{MR4045309},  for any $c\in KO_n^{lf}(\uspace_G)$, there is an element $(M, f)\in \Omega^{spin, lf}_\ast(B_G) $ such that $M$ is a noncompact spin manifold and $f$ is a proper map from $M$ to $\uspace_G$ satisfying the following: 
	\begin{enumerate}[(a)]
		\item $f_\ast[D_M] = c$, where $D_M$ is the Dirac operator on $M$;
		\item the inverse images $(M_i, \partial M_i) = f^{-1}(Y_i, \partial Y_i)$ are compact manifolds with boundary such that the induced maps $\pi_1(M_i)\to \pi_1(Y_i)$ and $\pi_1(\partial M_i) \to \pi_1(\partial Y_i)$ are all isomorphisms;
		\item the induced maps $\pi_1(A_i)\to \pi_1(Y_{i,i+1})$ are  isomorphisms for all $i\geq 0$, where $A_i = M_{i+1}- \interior{M}_{i}$ is the annulus region between $\partial M_{i}$ and $\partial M_{i+1}$.
	\end{enumerate}
 By Lemma \ref{lm:surgery2}, these $A_i$'s   satisfy condition (4). This finishes the proof. 
\end{proof}

We will also need an extension  result for Riemannian metrics on certain types of cobordims (Proposition \ref{prop:ext}).  First let us recall the following extension lemma of Shi, Wang and Wei \cite[Lemma 2.1]{shi-wang-wei}.

\begin{lemma}[{\cite[Lemma 2.1]{shi-wang-wei}}]\label{lm:sww}
		Let $\Sigma$ be a closed smooth manifold. Suppose $h_1$ and $h_0$ are two smooth Riemannian metrics on $\Sigma$ such that $h_1 < h_0$. Then for any constant $k>0$ and $m\in \mathbb R$, there exists a smooth Riemannian metric $g$ on the cylinder $\Sigma\times [0, 1]$ such that 
	\begin{enumerate}[$(1)$]
		\item $g$ extends $h_0$ and $h_1$, that is, \[ g|_{\Sigma\times \{0\}}  = h_0 \textup{ and } g|_{\Sigma\times \{1\}}  = h_1,  \]
		\item $\Sc(g) \geq k$,
	    \item the mean curvature at $0$-end of the cylinder is bounded below by $m$, that is, 
	    \[  H_{g}(\Sigma\times \{0\})_x\geq m  \]
	    for all $x\in  \Sigma\times \{0\}$.
	\end{enumerate}
\end{lemma}
Here the mean curvature $H_{g}(\Sigma\times \{0\})_x$ is calculated with respect to the outer normal vector. In particular, our convention is that the mean curvature of the standard $n$-dimensional sphere is positive when viewed as the boundary of the standard $(n+1)$-dimensional Euclidean ball.  

Combining the above extension lemma of Shi, Wang and Wei with  the surgery theory  for positive scalar curvature metrics  of Gromov-Lawson  \cite{MGBL80b} and Schoen-Yau \cite{RSSY79b}, we have the following proposition. 

\begin{proposition}\label{prop:ext}
	Let $Z$ be a cobordism between two closed smooth  manifolds $\Sigma_1$ and $\Sigma_2$ such that $Z$  is obtained from $\Sigma_2\times I$ by attaching handles of index $\geq 2$. Given any smooth Riemannian metric $h$ on $\Sigma_1$, for any constants $k>0$ and $m \in \mathbb R$, there exists a smooth Riemannian metric $g$ on $Z$ such that 
	\begin{enumerate}[$(1)$]
		\item $g$ extends $h$, that is, \[ g|_{\Sigma_1} = h  \]
		\item $\Sc(g) \geq k$,
		\item $ H_{g}(\Sigma_1)_x\geq m$, 
		for all $x\in  \Sigma_1$.
	\end{enumerate}
\end{proposition} 
\begin{proof}
 Consider the cylinder $\Sigma_1\times I$. Let us equip $\Sigma_1\times \{0\}$ with the metric $h$, and $\Sigma_1\times \{1\}$ with any Riemannian metric $h_1$ such that $h_1<h$ (e.g. $h_1 = \frac{h}{2})$.  Let $g_0$ be a Riemannian metric on $\Sigma_1\times I$ delivered by  Lemma \ref{lm:sww} for the constants $(k+1)$ and $m$.
 
 Now let us consider the cobordism 
 \[ W = Z \cup_{\Sigma_2} (-Z)\]
 obtained by gluing $Z$ with its opposite $-Z$ along the boundary component $\Sigma_2$. 
 \begin{claim}\label{claim:codimensionthree}
	$W$ is obtained from the cylinder $\Sigma_1\times I$ via surgeries of codimension $\geq 3$.
 \end{claim}	
 Indeed, the space $Z\times [0, 1]$ is  a cobordism between 
 \[ Z\times \{0\}\cup \Sigma_{2}\times [0, 1] \cup Z\times \{1\} \cong W  \textup{ and } \Sigma_1\times [{\scriptstyle  -\frac{1}{2}}, { \scriptstyle \frac{1}{2}}] \cong \Sigma_1\times I. \]  Since $Z$ is obtained from $\Sigma_2\times I$ by attaching handles of index $\geq 2$, it is not difficult to see that $W$ is obtained from the cylinder $\Sigma_1\times I$ via surgeries of codimension $\geq 3$. More precisely,  recall that attaching a handle of index $(p+1)$ to $\Sigma_1\times I$ is given by 
 \[ (\Sigma_1\times I) \cup_{\sphere^p \times \disk^{q}} (\disk^{p+1}\times \disk^q). \]
 where $\sphere^p$ is a $p$-dimensional sphere in $\Sigma_1\times \{0\} = \Sigma_1$ and $\sphere^p \times \disk^{q}$ is a tubular neighborhood of $\sphere^p$ in $\Sigma_1$. For brevity, let us denote the resulting space from the above handle attaching construction by $Y$. Note that $Y$ has two boundary components: one of them is $\Sigma_1$ and the other one is denoted by $\Sigma'$.   Let $Y\cup_{\Sigma'} Y$ be the space obtained from two copies of $Y$ glued  along $\Sigma'$. We claim that $Y\cup_{\Sigma'} Y$ is obtained from  $(\Sigma_1\times I)\cup_{(\Sigma_1\times \{0\}) } (\Sigma_1\times I)$ by performing a codimension  $(q+1)$ surgery. Indeed, let  $\sphere^p$ be the $p$ dimensional sphere in $\Sigma_1\times \{0\}$ from above. Its tubular neighborhood in $(\Sigma_1\times I)\cup_{(\Sigma_1\times \{0\}) } (\Sigma_1\times I)$ is  $\sphere^p\times \disk^{q+1}$. It is not difficult to see that the space $Y\cup_{\Sigma'} Y$ is obtained from $(\Sigma_1\times I)\cup_{(\Sigma_1\times \{0\}) } (\Sigma_1\times I)$
 by removing the interior of $\sphere^p\times \disk^{q+1}$ and gluing a copy of 
 \[ \disk^{p+1} \times (\disk^q\cup_{\partial \disk^q } \disk^q )  \cong \disk^{p+1} \times \sphere^q\]
 to $\sphere^p\times \sphere^{q} = \partial (\sphere^p\times \disk^{q+1})$. By assumption $Z$ is obtained from $\Sigma_2\times I$ by attaching handles of index $\geq 2$, or equivalently $Z$ is obtained from $\Sigma_1\times I$ by attaching handles of index $\leq (\dim Z - 2)$. It follows from the above discussion that $W$ is obtained from $ (\Sigma_1\times I)\cup_{(\Sigma_1\times \{0\}) } (\Sigma_1\times I) \cong \Sigma_1 \times I$ via surgeries of codimension $\geq 3$. This proves the claim.

 Now by surgery theory  for positive scalar curvature metrics  of Gromov-Lawson  \cite{MGBL80b} and Schoen-Yau \cite{RSSY79b}, it follows that there exists a positive scalar curvature metric $g_1$ on $W$. Furthermore, the construction of $g_1$ on $W$ (as in  \cite{MGBL80b}) in fact shows that for any $\varepsilon >0$, there exists a Riemannian metric $\bar{g}$ on $W$ such that $\Sc(\bar{g}) \geq k+1-\varepsilon$. Therefore, without loss of generality, let us assume $\Sc(g_1)\geq k$. Moreover, as all surgeries are performed away from the boundary, $g_1$ and $g_0$ coincide near the boundary. In particular, we still have 
 \[   g_1|_{\Sigma_1\times \{0\}} = h \textup{ and }  H_{g_1}(\Sigma_1\times \{0\})_x\geq m, \textup{ 
 for all } x\in  \Sigma_1.\] 
Let $g = g_1|_Z$ be the restriction of the Riemannian metric $g_1$ on  $Z$. Then $g$ satisfies all the required properties. This finishes the proof. 
	
\end{proof}	

\begin{remark}
	Note that in Proposition \ref{prop:ext} above, we do not have too much control of the Riemannian metric at $\Sigma_2$, neither the mean curvature of $\Sigma_2$. 
\end{remark}

Another ingredient needed for the proof of our main theorem is the following gluing lemma of Miao \cite[Section 3]{MR1982695}. See also \cite[Section 11.5]{MR3816521}. 

\begin{lemma}\label{lm:glue}
	Let $(X_1, g_1)$ and $(X_2, g_2)$ be two Riemannian manifolds with compact boundary. Suppose $\varphi\colon \partial X_1 \to \partial X_2$ is an isometry with respect to the induced metrics on $X_1$ and $X_2$. If 
	\[  H_{g_1}(\partial X_1)_x + H_{g_2}(\partial X_2)_{\varphi(x)} > 0\]
	for all $x\in \partial X_1$, then the natural continuous Riemannian metric $g_1\cup g_2$ on $X_1\cup_\varphi X_2$ can be approximated by smooth Riemannian metrics $g_{\varepsilon}$ with their scalar curvatures bounded from below by the scalar curvature\footnote{For $x \in \partial X_1 \cong \partial X_2$, this means $\Sc(g_\varepsilon)_x \geq \max \{\Sc(g_1)_x, \Sc(g_2)_{\varphi(x)}\}$. } of  $g_1\cup g_2$. Furthermore, for any $\varepsilon>0$,  $g_\varepsilon$ can be chosen to coincide with $g_1\cup g_2$ away from $\varepsilon$-neighborhood of $\partial X_1$ in $X_1\cup_\varphi X_2$. 
\end{lemma}

We now combine the above ingredients to give a negative answer to Gromov's compactness question when interpreted in its complete generality.

\begin{theorem}\label{thm:main}
	There exists a noncompact smooth spin manifold $M$ such that for any given compact subset $V\subset M$ and any $\rho>0$, there exists a (non-complete) Riemannian metric on $X$ with scalar curvature $\geq 1$ and the closed $\rho$-neighborhood $N_\rho(V)$ of $V$ in $M$ is compact, but $M$ itself does not admit a complete Riemannian metric with uniformly positive scalar curvature. 
\end{theorem}
\begin{proof}
	Let $\{G_i, \varphi_i\}$ be the inverse system from \eqref{eq:sys}. Let $\uspace_G$ be the corresponding space equipped with the exhaustion $\{Y_i\}$ as before. 
	
	By  Proposition \ref{prop:nonzero}, $\sideset{}{^1}\varprojlim KO_3(Y_j, \partial Y_j)$ is nontrivial. Let $c$ be a nonzero element in  $\sideset{}{^1}\varprojlim KO_3(Y_j, \partial Y_j)\subset  KO^{lf}_3 (\uspace_G)$. By Theorem \ref{thm:realization},  there is an element $(M, f)\in \Omega^{spin, lf}_\ast(\uspace_G) $ such that $M$ is a noncompact spin manifolds of dimension $\geq 6$ and $f$ is a proper map from $M$ to $\uspace_G$ satisfying the following: 
	\begin{enumerate}[$(1)$]
			\item $f_\ast[D_M] = c$;
		\item the inverse images $(M_i, \partial M_i) = f^{-1}(Y_i, \partial Y_i)$ are compact manifolds with boundary such that the induced maps $\pi_1(M_i)\to \pi_1(Y_i)$ and $\pi_1(\partial M_i) \to \pi_1(\partial Y_i)$ are all isomorphisms;
		\item the induced maps $\pi_1(A_i)\to \pi_1(Y_{i,i+1})$ are  isomorphisms for all $i\geq 0$, where $A_i = M_{i+1}- \interior{M}_{i}$ is the annulus region between $\partial M_{i}$ and $\partial M_{i+1}$;
		\item each $A_i$  is obtained from $(\partial M_{i+1})\times I$ by attaching handles of index $\geq 2$, or  equivalently $\partial M_{i+1}$ is obtained from $\partial M_{i}$ via surgeries of codimension $\geq 2$.
	\end{enumerate}	

We have the following commutative diagram (cf. \cite{Guentner-Yu}): 
\[  \begin{tikzcd}[column sep=1.5em]
	0 \arrow[r] &  \sideset{}{^1}\varprojlim KO_{n+1}(M_i, \partial M_i)  \arrow[r] \arrow[d] &  KO_n^{lf} (M)  \arrow[r] \arrow[d]& \varprojlim KO_n(M_i, \partial M_i) \arrow[r]  \arrow[d] &  0 \\
	0 \arrow[r] &  \sideset{}{^1}\varprojlim KO_{n+1}(Y_i, \partial Y_i)  \arrow[r] \arrow[d, "\alpha"] &  KO_n^{lf} (\uspace_G)  \arrow[r] \arrow[d, "\sigma"]& \varprojlim KO_n(Y_i, \partial Y_i) \arrow[r]  \arrow[d, "\beta"] &  0 \\
	0 \arrow[r] &  \sideset{}{^1}\varprojlim KO_{n+1}(D_i^\ast) \arrow[r] &  KO_n (\ualg(\uspace_G))  \arrow[r]& \varprojlim KO_n(D_i^\ast) \arrow[r]  &  0 
\end{tikzcd} \]
where  $\ualg(\uspace_G)$ is defined as in  Definition \ref{def:ualg} and 	\[ D_i^\ast \coloneqq \varinjlim_{j>i} C^\ast_{\max}(\pi_1(Y_j), \pi_1(Y_{ij})) \otimes \mathcal K \]   
with $Y_{ij} = Y_j - \interior{Y}_i$. By construction, $Y_{ij}$ deformation retracts to $\partial Y_i = BG_i$ for all $j>i$. By the homotopy invariance of the fundamental group, it follows that 
\[ D_i^\ast \coloneqq  C^\ast_{\max}(\pi_1(Y_i), \pi_1(\partial Y_{i})) \otimes \mathcal K. \] 
By construction, each $Y_i$ is contractible. Moreover,  $\partial Y_i = BG_i$ with $G_i = \mathbb Z^2\ast F_i$ where $F_i$ is a finitely generated free group. Consequently, the (relative) Baum-Connes assembly map 
\[   \beta\colon  KO_n(Y_i, \partial Y_i) = KO_n(\{pt\}, BG_i) \to  KO_n(D_i^\ast) =  KO_n(C^\ast_{\max}(\{e\}, G_i)) \] 
is an isomorphism.  In particular, it follows that the maps $\alpha, \beta$ and $\sigma$ in the above commutative diagram are isomorphisms.  

Since $c$ is an element in $\sideset{}{^1}\varprojlim KO_3(Y_j, \partial Y_j)$, it follows from the above commutative diagram that its image  in  $ \varprojlim KO_2(Y_i, \partial Y_i)$ equals $0$, hence $0$ in $\varprojlim KO_2(D_i^\ast)$. In particular, it is $0$ in each $KO_2(D_i^\ast)$. By the discussion following the unstable relative Gromov-Lawson-Rosenberg conjecture (Conjecture \ref{conj:rGLR}), since $G_i = \mathbb Z^2\ast F_i$ with $F_i$ a finitely generated free group,  each $M_i$ admits a positive scalar curvature metric that is collared near $\partial M_i$.

Note that any compact subset $V\subset M$ is contained in some $M_i$.  Since each $M_i$ has a metric of positive scalar curvature which has product structure near the boundary, by stretching the collar neighborhood of the boundary (which does not change the scalar curvature) if necessary,  it follows that for any  compact subset $V\subset M$ and any $\rho>0$, there exists  $M_i$ equipped with a metric $g_i$ of positive scalar curvature such that $\Sc(g_i)\geq 1$ and  the closed $\rho$-neighborhood $N_\rho(V)$ of $V$  is contained in $M_i$, hence $N_\rho(V)$ is compact. 

Now we show that the Riemannian metric $g_i$ on $M_i$ can be extended to an (non-complete) Riemannian metric $g$ on $M$ with $\Sc(g)\geq 1$. Let us denote by $h_i$ the restriction of $g_i$ on $\partial M_i$. Let $m_i$ be a real number such that the mean curvature of $\partial M_i$ satisfies \[ H_{g_i}(\partial M_i)_x >  -m_i  \]
 for all $x\in \partial M_i$.  By our choice of $M$,  the  annulus region $A_{i} = M_{i+1} -\interior{M}_{i} $ obtained from $\partial M_{i+1}\times I$ by attaching handles of index $\geq 2$. Then it follows from Proposition \ref{prop:ext} that  there exists a smooth Riemannian metric $\tilde {g}_i$ on $A_{i}$ such that 
 \begin{enumerate}[$(1)$]
 	\item $\tilde {g}_i$ extends $h_i$, that is, \[ \tilde {g}_i|_{\partial M_i } = h_i  \]
 	\item $\Sc(\tilde {g}_i) \geq 1$,
 	\item  $ H_{g}(\partial M_i)_x\geq m_i +1$, 
 	for all $x\in  \partial M_i$.
 \end{enumerate}
By applying Miao's gluing  lemma (Lemma \ref{lm:glue}), we obtain a smooth Riemannian metric $g_{i+1}$ on $M_{i+1}= M_{i}\cup_{\partial M_i} A_{i}$ such that 
   \begin{enumerate}[(i)]
  	\item ${g}_{i+1} = g_i$ on $M_i - N_{\varepsilon_i}(\partial M_i)$, where $\varepsilon_i$ is a sufficiently small\footnote{Here ``sufficiently small" means that the $\varepsilon_i$-neighborhood $N_{\varepsilon_i}(\partial M_{i})$ is disjoint from $M_{i-1}$.   }  positive number, 
  	\item $\Sc(\tilde {g}_{i+1}) \geq 1$.
  \end{enumerate}
Now repeat the above argument inductively on each $A_{j}$ for $j > i+1$. In the end, we obtain  a (non-complete) Riemannian metric on $M$ with scalar curvature $\geq 1$ such that the closed $\rho$-neighborhood $N_\rho(V)$ of $V$ in $M$ is compact.

Now we shall complete the proof by showing that $M$ itself does not have a complete  metric of uniformly positive scalar curvature. Let us prove this by contradiction. Assume to the contrary that $M$ admits a complete metric of uniformly positive scalar curvature. We can choose a metric on $\uspace_G$ under which  $\{Y_i, Y_{ij}\}$ is an admissible exhaustion of $\uspace_G$ and  the proper map $f\colon M \to \uspace_G$ is a continuous  proper coarse map. By \cite[Theorem 3.3]{MR4045309}, the higher index of $D_M$ under the index map 
\[ \sigma\colon KO^{lf}_\ast(\uspace_G)  \to KO_\ast(\ualg(\uspace_G))  \] 
vanishes, that is, $\sigma(f_\ast[D_M]) = 0 \in KO_\ast(\ualg(\uspace_G)).$ This contradicts the fact that $\sigma$ is an isomorphism and the   assumption that $f_\ast[D_M]  = c\neq 0$. 
 This finishes the proof. 
\end{proof}

\begin{remark}
In the above proof, we used Proposition \ref{prop:ext} and  Miao's gluing  lemma (Lemma \ref{lm:glue}) to construct an incomplete  Riemannian metric $g_{\rho}$ on $M$ with scalar curvature $\geq 1$ such that the closed $\rho$-neighborhood $N_\rho(V)$ of $V$ in $M$ is compact. A key ingredient in the proof of Proposition \ref{prop:ext} is Claim \ref{claim:codimensionthree} which roughly says that if $\Sigma_2$ is obtained from $\Sigma_1$ via surgeries of codimension $\geq 2$ and $Z$ is the cobordism representing the trace of the surgeries, then $Z\cup_{\Sigma_2} Z$ is obtained from $\Sigma_1\times I$ via surgeries of codimension $\geq 3$. In fact, one can also   construct such an incomplete metric $g_\rho$ on $M$ by directly  applying Claim \ref{claim:codimensionthree} together with the surgery theory of positive scalar curvature Gromov-Lawson  \cite[Theorem A]{MGBL80b} and Schoen-Yau \cite[Corollary 6]{RSSY79b}. We thank Bernhard Hanke for pointing this out to us. Here is a sketch of this alternative argument. The first step is the same as in the proof. For any  compact subset $V\subset M$ and any $\rho>0$, there exists  $M_i$ equipped with a metric $g_i$ of positive scalar curvature such that $\Sc(g_i)\geq \frac{3}{2}$ and  the closed $\rho$-neighborhood $N_\rho(V)$ of $V$  is contained in $M_i$, hence $N_\rho(V)$ is compact. We need to extend the Riemannian metric $g_i$ on $M_i$  to an (incomplete) Riemannian metric $g_\rho$ on $M$ with $\Sc(g_\rho)\geq 1$. By our choice of $M$,  the  annulus region $A_{i} = M_{i+1} -\interior{M}_{i} $ is obtained from $\partial M_{i+1}\times I$ by attaching handles of index $\geq 2$. Consider a new manifold $  X = M_{i+1}\cup_{\partial M_{i+1}} (-A_{i})$ obtained by gluing another copy of $A_i$ to $M_{i+1}$ along the common boundary $\partial M_{i+1}$. By Claim \ref{claim:codimensionthree}, $A_i\cup_{\partial M_{i+1}} (-A_{i})$ is obtained from $\partial M_{i}\times I$ via surgeries of codimension $\geq 3$. Since $g_i$ is a positive scalar curvature metric on $M_i$ such that $\Sc(g_i)\geq 1$, if we view $M_i \cong M_i\cup_{\partial M_{i}} (\partial M_{i}\times I) $, it follows from the surgery theory of positive scalar curvature that $X$ admits a positive scalar curvature metric $g_{i+1}$ such that $g_{i+1} = g_i$ away from  $A_i\cup_{\partial M_{i+1}} (-A_{i}) \subset X$ and $\Sc(g_{i+1}) \geq \frac{3}{2}-\varepsilon_i$ for some arbitrarily small $\varepsilon_i$. We denote the restriction of $g_{i+1}$ on $M_{i+1} \subset X$ still by $g_{i+1}$. Now we repeat this argument for each $A_{k} = M_{k+1} -\interior{M}_{k}$ and finally obtain 
an incomplete  Riemannian metric $g_{\rho}$ on $M$ with scalar curvature $\geq 1$ such that the closed $\rho$-neighborhood $N_\rho(V)$ of $V$ in $M$ is compact. 
\end{remark}

\section{Positive answers to Gromov's compactness question}\label{sec:pos}

In our construction of negative examples to Gromov's compactness question (cf. Theorem \ref{thm:main}), we have seen that the non-vanishing of a certain $\sideset{}{^1}\varprojlim$ index is an obstruction to the existence of complete Riemannian metrics of uniformly positive scalar curvature on some noncompact spin manifolds. This suggests that Gromov's compactness question has a positive answer if in addition one assumes that an appropriate $\sideset{}{^1}\varprojlim$ index  vanishes. More precisely, we have the following conjecture (Conjecture \ref{conj:limcomp}) based on Gromov's compactness quesiton.

\begin{conjecture}\label{conj:limcomp}
	Let $X$ be a smooth spin manifold.  Suppose for any given compact subset $V\subset X$ and any $\rho>0$, there exists a (non-complete) Riemannian metric on $X$ with scalar curvature $\geq 1$ such that the closed $\rho$-neighborhood $N_\rho(V)$ of $V$ in $X$ is compact. If the $\sideset{}{^1}\varprojlim$ higher index of the Dirac operator of $X$ vanishes,  then $X$ admits a complete Riemannian metric with scalar curvature $\geq 1$. 
\end{conjecture}
The  $\sideset{}{^1}\varprojlim$ higher index was introduced in Definition \ref{def:limone}.  The assumption ``for any given compact subset $V\subset X$ and any $\rho>0$, there exists a (non-complete) Riemannian metric on $X$ with scalar curvature $\geq 1$ such that the closed $\rho$-neighborhood $N_\rho(V)$ of $V$ in $X$ is compact" implies that the $\sideset{}{^1}\varprojlim$ higher index of the Dirac operator of $X$ is well-defined (cf. the discussion before Definition \ref{def:limone}). 

Due to the failure of the unstable Gromov–Lawson-Rosenberg conjecture (\cite[Example 2.2]{MR1632971}), hence the failure of the unstable relative Gromov–Lawson-Rosenberg conjecture, Conjecture \ref{conj:limcomp} is most likely not true in its complete generality. Rather one should interpret it as a guide towards the correct compactness statement for positive scalar curvature metrics on open manifolds. On the other hand, as supporting evidence for the philosophy behind Conjecture \ref{conj:limcomp},  we verify the conjecture for a class of $1$-tame  spin manifolds.\footnote{With extra geometric conditions, Theorem \ref{thm:pos} of this section still holds for nonspin manifolds.}

\begin{definition}\label{def:tame}
	A manifold $M$ is said to be $1$-tame if there is a sequence of codimension zero compact submanifolds $\{K_i\}_{i \geq 1}$ (with boundary) such that 
	\begin{enumerate}
		\item $K_i\subset K_{i+1}$ and $M = \cup_{i\geq 1} K_i$, 
		\item  all $M- K_{i}$ have the same finite number of connected components. If we denote the connected components of $M- K_{i}$ by $\{P_{i,j}\}_{1\leq j \leq m}$ with $P_{i+1, j} \subset P_{i, j}$, then the natural  homomorphism $\pi_1(P_{i+1, j}) \to   \pi_1(P_{i, j})$ induced by the inclusion  is  an isomorphism for all $i$ and $j$. 
	\end{enumerate}
\end{definition}
For simplicity, we define $\pi_1(M- K_i)$ to be the disjoint union $\coprod_j \pi_1(P_{i,j})$.  We say  the  map   $\pi_{1}(M - K_{i+1}) \to \pi_1(M- K_i)$ is an isomorphism if the condition (2) above is satisfied. In this case, we call  $\pi_1(M- K_i)$ the fundamental groupoid at infinity of $M$, denoted by $\pi_1^\infty(M) = \coprod_{\alpha = 1}^\ell G_\alpha$. Throughout this section, we assume each $G_\alpha$ is finitely presented. 

\begin{remark}\label{rmk:tame}
By a theorem of Siebenmann \cite[Theorem 3.10]{MR2615648}, for any smooth open manifold $M$ with $\dim M \geq 5 $, if $M$ is $1$-tame, then we can choose  $\{K_i\}_{i \geq 1}$ in Definition \ref{def:tame} so that the inclusion $\partial K_i\to (M - \interior{K}_i)$ induces an isomorphism  
$\pi_1(\partial K_i) \to \pi_1(M - \interior{K}_i) $  for each $i\geq 1$. As a consequence, if we denote the annulus region between $\partial K_i$ and $\partial K_{i+1}$ by $A_i$, then  the inclusions $\partial K_i \to A_{i}$ and $\partial K_{i+1}\to A_{i}$ induce isomorphisms (component-wise) on $\pi_1$  \cite[Lemma 3.12]{MR2615648}. 
\end{remark}

The following theorem gives a positive answer to Gromov's compactness question for $1$-tame spin manifolds, provided that  the unstable relative Gromov-Lawson-Rosenberg conjecture holds for the relevant fundamental groups. 

\begin{theorem}\label{thm:pos}
	Let $M$ be a noncompact $1$-tame spin manifold of dimension $ n\geq 6$. Let $\Gamma = \pi_1(M) $ and $G = \pi_1^\infty(M)$. Assume that the unstable relative Gromov-Lawson-Rosenberg conjecture holds for the pair $(\Gamma, G)$.  Suppose for any given compact subset $V\subset M$ and any $\rho>0$, there exists a (non-complete) Riemannian metric on $M$ with scalar curvature $\geq 1$ such that the closed $\rho$-neighborhood $N_\rho(V)$ of $V$ in $M$ is compact. Then $M$ admits a complete Riemannian metric of uniformly positive scalar curvature. 
\end{theorem}
\begin{proof}
	Let $\{K_i\}$ be a sequence of codimension zero compact submanifolds of $M$ satisfying the properties given in Definition \ref{def:tame}. By Remark \ref{rmk:tame}, without loss of generality, we assume the inclusion $\partial K_i\to (M - \interior{K}_i)$ induces an isomorphism  
	$\pi_1(\partial K_i) \to \pi_1(M - \interior{K}_i) $  for each $i\geq 1$. If we denote the annulus region between $\partial K_i$ and $\partial K_{i+1}$ by $A_i$, then  the inclusions $\partial K_i \to A_{i}$ and $\partial K_{i+1}\to A_{i}$ induce isomorphisms (component-wise) on $\pi_1$, cf. \cite[Lemma 3.12]{MR2615648}.

	Choose  $\rho >0$ to be sufficiently large and $V = K_1$. By the assumption, there exists a Riemannian metric $g_0$ on $M$ such that $\Sc(g_0)\geq 1$ and the closed $\rho$-neighborhood $N_\rho(V)$ of $V$ in $M$ is compact. The Riemannian metric $g_0$ is incomplete in general. But we can always modify the metric on  $M - N_\rho(V)$ while keeping the metric on $N_\rho(V)$ fixed so that the resulting new metric $h$ is complete Riemannian metric on $M$ such that $\Sc(h)\geq 1$ on $N_\rho(V)$.

	Let us write $M - V$ as a disjoint union of connected components 
	\[ M - V = \sqcup_{\alpha=1}^\ell Y_\alpha \] and  similarly $\partial V = \sqcup_{\alpha=1}^\ell \partial Y_\alpha$.   By the discussion above,   $\pi_1(\partial Y_\alpha)\to  \pi_1(Y_\alpha)$ is an isomorphism.  For simplicity, we shall write $G = \sqcup_{\alpha = 1}^\ell G_\alpha$ with $G_\alpha = \pi_1(\partial Y_\alpha)$.  
	Let us still denote by $h$   the restriction of the metric $h$ on $Y_\alpha$.  Each inclusion $\iota_\alpha\colon Y_\alpha\to M$ induces a homomorphism $\iota_\alpha\colon G_\alpha \to \Gamma$. Consequently, we have a $C^\ast$-algebra homomorphism 
	\[   \iota_\ast = \oplus_{\alpha=1}^\ell  (\iota_\alpha)_\ast\colon \bigoplus_{\alpha = 1}^\ell  C_{\max}^\ast(G_\alpha)  \to C^\ast_{\max}(\Gamma). \]
	We define  $C^\ast_{\max}(\Gamma, G)$ to be  the $7$th suspension $\sus^7C_{\iota_\ast} \cong C_0(\mathbb R^7)\otimes C_{\iota_\ast} $ of  the mapping cone $C^\ast$-algebra  $C_{\iota_\ast}$. With the above setup, the relative higher index $\ind_{\Gamma, G}(D_V)$ of the Dirac operator $D_V$ of $V$ is an element in  $KO_n(C^\ast_{\max}(\Gamma, G))$.  See Section \ref{sec:relative} for more details on relative higher index.  As long as $\rho$ is sufficiently large,  it follows from Theorem \ref{thm:quanvanish} that  $\ind_{\Gamma, G}(D_V)$  vanishes in  $KO_n(C^\ast_{\max}(\Gamma, G))$.

    By assumption, the unstable relative Gromov-Lawson-Rosenberg conjecture holds for the pair $(\Gamma, G)$. It follows that\footnote{Here we have implicitly used the fact that  $\pi_1(\partial Y_\alpha)\to  \pi_1(Y_\alpha)$ is an isomorphism.} $V$ admits a Riemannian metric $g_V$ of positive scalar curvature $\geq 2$ that is collared near the boundary. Now it only remains to show that the metric $g_V$ on $V$ extends to a complete Riemannian metric on $M$ with scalar curvature $\geq 1$. 
    
    Consider the cobordism $A_i$ between $\partial K_i$ and $\partial K_{i+1}$. 
    Since the inclusions $\partial K_i \to A_{i}$ and $\partial K_{i+1}\to A_{i}$ induce isomorphisms (component-wise) on $\pi_1$. It follows from Lemma \ref{lm:surgery} that $A_i$ can be viewed as  the trace of surgeries of codimension $\geq 3$ from $\partial K_{i}$ to $\partial K_{i+1}$.     By the surgery theory for positive scalar curvature of Gromov-Lawson  \cite[Theorem A]{MGBL80b} and Schoen-Yau \cite[Corollary 6]{RSSY79b}, we can  extend the metric $g_V$ on $V = K_1$ to a Riemannian metric $g_2$ on $K_2$ such that $\Sc(g_2) \geq 2 - \frac{1}{2}$ and $g_2$ is collared near $\partial K_2$. By stretching out\footnote{Note that such a stretching does not change the scalar curvature lower bound of $g_2$ since $g_2$ has product structure near $\partial K_2$.} a small  tubular neighborhood of $\partial K_2$ if necessary, we can assume that $\mathrm{dist}(\partial K_1, \partial K_2)\geq 2$. Now by the surgery theory for positive scalar curvature, we can inductively extend  the metric $g_i$ on $K_i$ to a Riemannian metric on $g_{i+1}$ on $K_{i+1}$ such that  $\Sc(g_{i+1}) \geq 2 - \sum_{\beta=1}^{i}\frac{1}{2^\beta}$, $g_{i+1}$ is collared near $\partial K_{i+1}$ and $\mathrm{dist}(\partial K_i, \partial K_{i+1})\geq i$. Hence by induction, we eventually obtain a complete Riemannian metric on $M$ with scalar curvature $\geq 1$. This finishes the proof.

\end{proof}

\end{document}